\documentclass[12pt]{article}
\usepackage{fancyhdr}

\usepackage{amsmath}
\usepackage{mathtools}
\usepackage{amsfonts}
\usepackage{amsthm}
\usepackage{amssymb}
\usepackage{color}
\usepackage{dsfont}
\usepackage{geometry} 
\newgeometry{tmargin=2.8cm, bmargin=3.5cm, lmargin=2cm, rmargin=2cm} 

\usepackage{color}

\title{Existence of renormalized solutions to elliptic equation \\ in Musielak-Orlicz space}

\usepackage{authblk}

\author[1]{Piotr Gwiazda\thanks{email address: p.gwiazda@mimuw.edu.pl}}
\author[1]{Iwona Skrzypczak\thanks{email address: iskrzypczak@mimuw.edu.pl}}
 \author[2]{Anna Zatorska--Goldstein\thanks{email address: azator@mimuw.edu.pl,\\  The research of P.G. has been supported by the NCN grant  no. 2014/13/B/ST1/03094. The research of A.Z.-G. has been supported by the NCN grant  no. 2012/05/E/ST1/03232 (years 2013 - 2017).  This work was partially supported by the Simons - Foundation grant 346300 and the Polish Government MNiSW 2015-2019 matching fund.}}
\affil[1]{\small
Institute of  Mathematics, Polish Academy of Sciences, \newline
ul. \'{S}niadeckich 8, 00-656 Warsaw, Poland
}
\affil[2]{\small
Institute of Applied Mathematics and Mechanics,
University of Warsaw, \newline
ul. Banacha 2, 02-097 Warsaw, Poland
}

\newcommand{\barint}{
         \rule[.036in]{.12in}{.009in}\kern-.16in
          \displaystyle\int  }
\def\R{{\mathbb{R}}}

\def\r{{\mathbb{R}}}
\def\N{{\mathbb{N}}}
\def\rn{{\mathbb{R}^{N}}}

\newcommand{\oN}{{\omega_{N}}}
\newcommand{\va} {\vec{a}}
\newcommand{\dep} {\delta}

\def\rp{{\mathbb{R}_{+}}}

\def\cA{{\cal{A}}}
\def\ve{{\varepsilon}}
\def\vr{{\varrho}}

\def\dm{{\underline{m}}}

\def\Mjd{{ M_j^\delta}}
\def\Mss{{(\Mjd(\xi))^{**}}}
\def\Msdx{{(\Mjd(\xi_\delta(x) ))^{**}}}
\def\Msd{{(\Mjd)^{**}}}

\def\ask{{{\cal A }_{s,k}}}
\def\asl{{{\cal A }_{s,l+1}}}
\def\al{{{\cal A }_{l+1}}}
\def\iO{{\int_{\Omega}}}
\def\iQd{{\int_{Q_j^\delta\cap\Omega} }}
\def\Qd{{Q_j^\delta}}
\def\tQd{{\widetilde{Q}_j^\delta}}
\def\iQdn{{\int_{Q_j^\delta\cap\{x:\xi_\delta(x)\neq 0\}} }}

\newtheorem{theo}{\bf Theorem}[section]

\newtheorem{lem}{\bf Lemma}[section]
\newtheorem{rem}{\bf Remark}[section]
\newtheorem{defi}{\bf Definition}[section]
\newtheorem{ex}{\bf Example}[section]

\newtheorem{prop}{\bf Proposition}[section]

\newcommand{\wt}{\widetilde}
 
\newcommand{\vp}{\varphi}

\newcommand{\dv}{\mathrm{div}}

\date{}
\begin{document}
\maketitle \sloppy

\thispagestyle{empty}


\parindent 1em

\begin{abstract}

We prove existence of renormalized solutions to   general nonlinear elliptic equation in~Musielak-Orlicz space avoiding growth restrictions. Namely, we consider \begin{equation*}
-\dv A(x,\nabla u)= f\in L^1(\Omega),
\end{equation*}
on a Lipschitz bounded domain in $\rn$. The growth of the monotone vector field $A$ is controlled by a generalized nonhomogeneous and anisotropic $N$-function  $M $. The approach does not require any particular type of growth condition of $M$ or its conjugate $M^*$ (neither $\Delta_2$, nor $\nabla_2$). The condition we impose is log-H\"older continuity of $M$, which results in good approximation properties of the space. The proof of the main results uses truncation ideas, the Young measures methods and monotonicity arguments.

\end{abstract}

\smallskip


\newpage
\section{Introduction}


Our aim is to find a way of proving the existence of renormalized solutions to a strongly nonlinear elliptic equation with $L^1$-data under minimal restrictions on the growth of the leading part of the operator.  We investigate operators $A$, which are monotone, but not necessarily strictly. The~modular function $M$, which controls the growth of the operator, is not assumed to be isotropic, i.e. $M=M(x,\xi)$ not $M=M(x,|\xi|)$. In turn, we can expect different behaviour of $M(x,\cdot)$ in various directions. We \textbf{do not} require  $M\in\Delta_2$, nor $M^*\in\Delta_2$, nor~any particular growth of $M$, such as $M(x,\xi)\geq c|\xi|^{1+\nu}$ for $\xi>\xi_0$.  The price we pay for relaxing the conditions on the growth is requirement of~log-H\"older-type regularity of the modular function (cf. condition (M)).

We study  the problem
\begin{equation}\label{intro:ell}\left\{\begin{array}{cl}
-\dv A(x,\nabla u)= f &\qquad \mathrm{ in}\qquad  \Omega,\\
u(x)=0 &\qquad \mathrm{  on}\qquad \partial\Omega,
\end{array}\right.
\end{equation}
where $\Omega$ is a bounded Lipschitz domain in $ \rn$, $N>1$, $f:\Omega\to\r$,  $f\in L^1(\Omega)$.

We consider $A$~belonging to an Orlicz class with respect to the second variable. Namely, we assume that function $A:\Omega\times\rn\to\rn$ satisfies the following conditions.
\begin{itemize}
\item[(A1)] $A$ is a Carath\'eodory's function;
\item[(A2)] There exists an $N$-function  $M:\Omega\times\rn\to\r$ and a constant $c_A\in(0,1]$ such that for all $\xi\in\rn$ we have
\[A(x,\xi)\xi\geq c_A\left(M(x,\xi)+M^*(x,A(x,\xi))\right),\]
where $M^*$ is conjugate to $M$ (see Definition~\ref{def:conj});
\item[(A3)] For all $\xi,\eta\in\rn$ and $x\in\Omega$ we have 
\[(A(x,\xi) - A(x, \eta)) \cdot (\xi-\eta)\geq 0.\]
\end{itemize}

Existence of solutions to~\eqref{intro:ell} is considered in
\[ V^M_0 =\{u\in W_0^{1,1}(\Omega):\ \nabla u\in L_M(\Omega;\rn)\}.\]
The space $L_M$ (Definition~\ref{def:MOsp}) is equipped with the modular function $M$ being an $N$-function (Definition~\ref{def:Nf}) controlling the growth of $A$,  cf.~(A2). 

Unlike other studies, instead of growth conditions we assume  regularity of~$M$.
\begin{itemize} 
\item[(M)] Suppose for every measurable set $G\subset\Omega$ and every $z\in\rn$ we have
\begin{equation}
\label{ass:M:int}\int_G M(x,z)dx<\infty.
\end{equation} Let us consider a family of $N$-dimensional cubes   covering the set $\Omega$. Namely, a family $\{\Qd\}_{j=1}^{N_\delta}$ consists of closed cubes of edge $2\delta$, such that  $\mathrm{int}\Qd\cap\mathrm{int} Q^\delta_i=\emptyset$ for $i \neq j$ and $\Omega\subset\bigcup_{j=1}^{N_\delta}\Qd$. Moreover, for each
cube $\Qd$ we define the cube $\tQd$ centered at the same point and with parallel corresponding edges of length $4\delta$.   Assume that there exist constants $a,b,c,\delta_0 >0$, such that for all $\delta<\delta_0$, $x\in\Qd$ and all $\xi\in\rn$ we have
\begin{equation}
\label{M2}  
\frac{M(x,\xi )}{\Mss} 
\leq c \left(1+ |\xi|^{-\frac{a}{\log(b\delta )}} \right),
\end{equation}
where \begin{equation}
\label{Mjd}\Mjd(\xi ):= \inf_{x\in \widetilde{Q}_j^\delta\cap\Omega}M(x,\xi ),
\end{equation} 
while $\Mss=({(\Mjd(\xi))}^*)^* $ is the greatest convex minorant of $\Mjd(\xi )$ (coinciding with the second conjugate cf. Definition~\ref{def:conj}).
\end{itemize}
 In further parts of the paper we describe the cases, when the above condition is not necessary. Let us only point out that to get (M) in the isotropic case, i.e. when we consider $M(x,\xi)=M(x,|\xi|)$, it suffices to assume log-H\"older-type condition with respect to $x$ \eqref{M2'},  cf.~Lemma~\ref{lem:Mass}.

We apply the truncation techniques. Let truncation $T_k(f)(x)$ be defined as follows\begin{equation}T_k(f)(x)=\left\{\begin{array}{ll}f & |f|\leq k,\\
k\frac{f}{|f|}& |f|\geq k.
\end{array}\right. \label{Tk}
\end{equation}

We call a function $u$ a renormalized solution to~\eqref{intro:ell}, when it satisfies the following conditions.\begin{itemize}
\item[(R1)] $u:\Omega\to\r$ is measurable and  for each $k>0$  \[T_k(u)\in V_0^M \cap L^\infty (\Omega),\qquad A(x,\nabla T_k(u))\in L_{M^*}(\Omega;\rn).\]
\item[(R2)] For every $h\in C^1_c(\R)$ and all $\varphi\in V_0^M\cap L^\infty (\Omega)$ we have
\[\int_\Omega A(x,\nabla u)\cdot\nabla(h(u)\varphi)dx=\int_\Omega fh(u)\varphi\,dx.\]
\item[(R3)] $\int_{\{l<|u|<l+1\}}A(x,\nabla u)\cdot\nabla u\, dx\to 0$ as $l\to\infty$.
\end{itemize}

Our main result reads as follows.
\begin{theo}\label{theo:main} Suppose $f\in L^1(\Omega)$, an $N$-function $M$ satisfies assumption (M) and function $A$ satisfies assumptions (A1)-(A3). Then there exists at least one renormalized weak solution to the problem \begin{equation*}
\left\{\begin{array}{cl}
-\dv A(x,\nabla u)= f &\qquad \mathrm{ in}\qquad  \Omega,\\
u(x)=0 &\qquad \mathrm{  on}\qquad \partial\Omega,
\end{array}\right.
\end{equation*} Namely, there exists $u \in V_0^M $ which satisfies (R1)-(R3). Moreover, $A(\cdot,\nabla u)\in L_{M^∗}(\Omega)$.
\end{theo}
\begin{ex} We give below pairs of functions  $M$ and $A$ satisfying conditions (M) and  (A1)-(A3), respectively.\begin{itemize}
\item Consider $M(x,\xi)=|\xi|^{p(x)}$ with log-H\"older $p:\Omega\to[p^-,p^+]$, where $p^-=\inf_{x\in\Omega}p(x)>1$ and $p^+=\sup_{x\in\Omega}p(x)<\infty,$ then $V_0^M=  W_0^{1,p(\cdot)}(\Omega)$ and we admitt $A(x,\xi)=|\xi|^{p(x)-2}\xi$ ($p(\cdot)$-Laplacian case) as well as \[A(x,\xi)=\alpha(x)|\xi|^{p(x)-2}\xi\quad\text{ with  }\quad 0<<\alpha(x)\in L^\infty(\Omega)\cap C(\Omega);\]
\item $M(x,\xi)=\sum_{i=1}^N|\xi_i|^{p_i(x)}$, where $\xi=(\xi_1,\dots,\xi_N)\in\rn$, log-H\"older functions $p_i:\Omega\to[p_i^-,p_i^+]$, $i=1,\dots,N$, where $p_i^-=\inf_{x\in\Omega}p_i(x)>1$ and $p_i^+=\sup_{x\in\Omega}p_i(x)<\infty,$ then $V_0^M= W_0^{1,\vec{p}(\cdot)}(\Omega)$ and we admitt  \[A(x,\xi)=\sum_{i=1}^N\alpha_i(x)|\xi|^{p_i(x)-2}\xi\quad\text{ with }\quad 0<<\alpha_i(x)\in L^\infty(\Omega)\cap C(\Omega).\]
\end{itemize}
\end{ex}

\begin{rem}[cf.~\cite{martin}] \rm When the modular function has a special form we can simplify our assumptions. In the case of $M(x,\xi)=M(x,|\xi|)$, via Lemma~\ref{lem:Mass}, we replace condition (M) in the above theorem by log-H\"older continuity of M, cf.~\eqref{M2'}. If $M$ has a form 
\[M(x,\xi)=\sum_{i=1}^jk_i(x)M_i(\xi)+M_0(x,|\xi|),\quad j\in\N,\]
instead of~whole (M) we assume only that $M_0$ is log-H\"older continuous~\eqref{M2'}, all $M_i$ for $i=1,\dots,j$ are $N$-functions and all $k_i$ are nonnegative and satisfy $\frac{k_i(x)}{k_i(y)}\leq C_i^{\log \frac{1}{|x-y|}}$ with $C_i>0$ for $i=1,\dots,j$.
\end{rem} 

\begin{rem} \rm 
Note that according to (A2) and the Fenchel-Young inequality we have
\[c_A\left(M(x,\xi)+M^*(x,A(x,\xi))\right)\leq A(x,\xi)\xi\leq  M(x,\xi)+M^*(x,A(x,\xi)) \]
satisfied with a certain $N$-function  $M:\Omega\times\rn\to\r$. This observation results in  $A(x,0)=0$. However, the framework admitts considering in (A2)
\[A(x,\xi)\xi\geq c_A\left(M(x,\xi)+M^*(x,A(x,\xi))\right)-k(x),\qquad 0\leq k(x)\in L^1(\Omega),\]
despite it does not imply $A(x,0)=0$.
\end{rem}
 
 The Musielak-Orlicz spaces equipped with the modular function satisfying $\Delta_2$-condition (cf.~Definition~\ref{def:D2}) have strong properties, however there is a vast range of $N$-functions not satisfying it, e.g.
\begin{itemize}
\item[i)] $M(x,\xi)=a(x)\left( \exp(|\xi|)-1+|\xi|\right)$;
\item[ii)] $M(x,\xi)= |\xi_1|^{p_1(x)}\left(1+|\log|\xi||\right)+\exp(|\xi_2|^{p_2(x)})-1$, when $(\xi_1,\xi_2)\in\R^2$ and $p_i:\Omega\to[1,\infty]$. It is a model example to imagine what we mean by anisotropic modular function. 
\end{itemize}
Nonetheless, our assumption that $M,M^*$ are $N$-functions (Definition~\ref{def:Nf})  in the variable exponent setting   restrict us to the case of $1<p_-\leq p(x)\leq p^+<\infty$.

\subsubsection*{State of art}

Existence to problems like~\eqref{intro:ell} is very well understood, when $A$ is independent of the spacial variable and has a polynomial growth. In~particular, there is vast literature for analysis of the case involving the $p$-Laplace operator $A(x,\xi)=|\xi|^{p-2}\xi$ and  problems  stated in the Lebesgue space setting (the modular  function is then $M(x,\xi)=|\xi|^p$). Let us note that  the variable exponent Lebesgue spaces (for $M(x,  \xi ) = | \xi |^{p(x)}$  with $1 < p_{\rm min} \leq p(x) \leq  p^{\rm max} < \infty $) are still reflexive. Despite the~methods of~analysis of problems in this setting are more advanced, they are in the same spirit.

Studies on renormalized solutions comes from DiPerna and Lions~\cite{diperna-lions} investigations on~the~Boltzmann equation.  In the elliptic setting the foundations of the branch were laid by Boccardo et.~al.~\cite{boc-g-d-m},  Dall'Aglio~\cite{dall} and Murat~\cite{murat}, providing results for operators with polynomial growth. Their generalisations to the variable exponent setting can be find in~\cite{andreianov,benboubker,wit-zim}. 

Investigations of nonlinear elliptic boundary value problems in~non-reflexive Orlicz-Sobolev-type setting was initiated by Donaldson~\cite{Donaldson} and continued by Gossez~\cite{Gossez2,Gossez3,Gossez}. For a~summary of~the~results we refer to~\cite{Mustonen} by Mustonen and Tienari. The generalization to the case of~vector Orlicz spaces with possibly anisotropic modular  function,  but independent of spacial variables was investigated in~\cite{Gparabolic}. 

The existence theory for problems in this setting arising  from fluids mechanics is developed from various points of view~\cite{gwiazda-non-newt,gwiazda-tmna,gwiazda2,Aneta}.  For the  recent existence results for elliptic problems we refer  to~\cite{renel2,renel1,Benb,renel3,Dong,fan12,le-ex,gwiazda-ren-ell,gwiazda-ren-cor,hhk,le-ex,liuzhao15}. In~\cite{fan12,hhk,liuzhao15} isotropic, separable and reflexive Musielak-Orlicz spaces are employed,~\cite{Benb} concerns anisotropic variable exponent spaces, \cite{Dong}~studies separable, but not reflexive Musielak-Orlicz spaces, while~\cite{le-ex} anisotropic, but separable and reflexive Orlicz spaces.
 Renormalized solutions to~elliptic problems in~Orlicz spaces are explored in~\cite{renel2,renel1,renel3}, while  in Musielak-Orlicz spaces in~\cite{gwiazda-ren-ell,gwiazda-ren-cor}.

\subsubsection*{Approximation in Musielak-Orlicz spaces} 
The highly challenging part of analysis in the general Musielak-Orlicz spaces is giving a relevant structural condition implying approximation properties of the space. However, we are equipped not only with the weak-* and strong topology of the gradients, but also with the intermediate one, namely - the modular topology.

 In the mentioned existence results even in the case, when the growth conditions imposed on~the~modular  function were given by a~general $N$-function, besides the growth condition on  $M^*$, also $\Delta_2$-condition on $M$ was assummed (which entails separability of~$L_{M^*}$, see~\cite{Aneta}). It results further in density of smooth functions in $L_M$ with respect to the weak-$*$ topology. In the case of~classical Orlicz spaces, the crucial density result was provided by Gossez~\cite{Gossez}. The improvement of this result for the vector Orlicz spaces was given in~\cite{Gparabolic}, while for the $x$--dependent log-H\"{o}lder continuous modular  functions in~\cite{BenkiraneDouieb}, developed in~\cite{GMWK,ASGcoll} and further in~\cite{ASGpara} in the case of log-H\"{o}lder continuous modular functions dependent on $x$, as well as on $t$. 
 
Let us discuss our assumption (M). First we shall stress that it is applied only in the proof of approximation result (Theorem~2.2). When we deal with the space equipped with the approximation properties, we can simply skip (M). Namely, this is the case e.g. of the following modular functions:
\begin{itemize}
\item $M(x,|\xi|)=|\xi|^p+a(x)|\xi|^q$, where $1\leq p<q$ and function $a$ is nonnegative a.e. in $\Omega$ and $a\in L^\infty(\Omega)$, covering the celebrated double-phase case~\cite{min-double-reg};
\item $M(x,\xi)=M_1(\xi)+a(x)M_2(\xi)$, where $M_1,M_2$ satisfy conditions  $\Delta_2$ and $\nabla_2$, moreover a~function $a$ is nonnegative a.e. in $\Omega$ and $a\in L^\infty(\Omega)$.
\end{itemize}
In the both above cases modular approximation sequence obtained in the spirit of Theorem~\ref{theo:approx} can be replaced by existence of a strongly converging  affine combination of the weakly converging sequence (ensured in any reflexive Banach space via Mazur's Lemma).

In the variable exponent case typical assumption resulting in approximation properties of the space is log-H\"older continuity of the exponent. In the isotropic case (when $M(x,\xi)=M(x,|\xi|)$) 
Lemma~\ref{lem:Mass} shows that to get (M), it suffices to   impose on $M$ continuity condition of log-H\"older-type with respect to $x$, namely for each $\xi\in\rn$ and $x,y,$ such that $|x-y|<\frac{1}{2}$ we have\begin{equation}
\label{M2'} \frac{M(x,\xi)}{M(y,\xi)}\leq\max\left\{ |\xi|^{-\frac{a_1}{\log|x-y|}}, b_1^{-\frac{a_1}{\log|x-y|}}\right\},\ \text{with some}\ a_1>0,\,b_1\geq 1.
\end{equation}  Note that condition~\eqref{M2'} 
 for $M(x,\xi)=|\xi|^{p(x)}$ relates to the log-H\"older continuity condition for the variable exponent $p$, namely there exists $a>0$, such that for $x,y$ close enough and each $\xi\in\rn$
\[|p(x)-p(y)|\leq \frac{a}{\log\left(\frac{1}{|x-y|}\right)}.\]
Indeed, 
\[ \frac{M(x,\xi)}{M(y,\xi)}= \frac{|\xi|^{p(x)}}{|\xi|^{p(y)}}=|\xi|^{p(x)-p(y)}\leq |\xi|^\frac{a}{\log\left(\frac{1}{|x-y|}\right)}=|\xi|^{-\frac{a}{\log {|x-y|} }}.\]

  There are several types of understanding generalisation of log-H\"older continuity to the case of general $x$-dependent isotropic modular functions (when $M(x,\xi)=M(x,|\xi|)$). The important issue is the interplay between types of continuity with respect to each of the variables separately. Besides our condition~\eqref{M2'} (sufficient for (M) via Lemma~\ref{lem:Mass}), we refer to the approaches of~\cite{hhk,hht} and~\cite{mmos:ap,mmos2013}, where the authors deal with the modular function of the form $M(x,\xi)=|\xi|\phi(x,|\xi|)$. We proceed without their doubling assumptions ($\Delta_2$). Since we are restricted to bounded domains, condition $\phi(x,1)\sim 1$ follows from our definition of $N$-function (Definition~\ref{def:Nf} ). As for the types of continuity,  in~\cite{mmos:ap,mmos2013} the authors restrict themselves to the case when $\phi(x,|\xi|)\le c \phi(y,|\xi|)$ when $|\xi|\in [1,|x-y|^{-n}].$ This condition implies~\eqref{M2'} and consequently~(M). Meanwhile in~\cite{hhk,hht}, the proposed condition yields  $\phi(x, b|\xi|)\le \phi(y,|\xi|)$ when $\phi(y,|\xi|)\in [1, |x-y|^{-n}],$ which does not imply~\eqref{M2'} directly. However, we shall mention that all three conditions are of the same spirit and balance types of continuity with respect to each of the variables separately.

\subsubsection*{Our approach}
 The challenges resulting from the lack of the growth conditions are significant and require precise handling with general $x$-dependent and anisotropic $N$-functions. The space we deal with is, in~general, neither separable, nor reflexive.  Resigning from imposing $\Delta_2$-condition on the conjugate of the modular function $M$ complicates understanding of the dual pairing. As a further consequence of relaxing growth condition, we cannot use classical results, such as the Sobolev embeddings or~the~Rellich-Kondrachov compact embeddings. We extend the main goal of~\cite{GMWK}, where the authors deal with bounded data.  Lack of~precise control on the growth of~the~ leading part of the operator, together with the low integrability 
   of~the~right-hand side results in noticeable difficulties in studies on convergence.
 
  Besides the refined version of approximation result of~\cite{GMWK} (Theorem~\ref{theo:approx}), we prove general modular Poincar\'{e}-type inequality (Theorem~\ref{theo:Poincare}). The main goal, i.e. the existence of renormalized solutions to  general nonlinear elliptic equation, is given in Theorem~\ref{theo:main}. Our methods leading to~this result are based on the scheme of~\cite{gwiazda-ren-ell,gwiazda-ren-cor}, i.e. we employ truncation arguments, the Minty-Browder monotonicity trick and the Young measures. However, unlike in the latter papers we put regularity restrictions on the modular function instead of the growth conditions.

\section{Preliminaries}
 
In this section we give only the general preliminaries concerning the setting. All necessary definitions and technical tools, as well as an introduction to the setting and general theorems are given in Appendix. 
 
\subsubsection*{Classes of functions}
 
\begin{defi}\label{def:MOsp} Let $M$ be an $N$-function (cf.~Definition~\ref{def:Nf}).\\ We deal with the three  Orlicz-Musielak classes of functions.\begin{itemize}
\item[i)]${\cal L}_M(\Omega;\rn)$  - the generalised Orlicz-Musielak class is the set of all measurable functions $\xi:\Omega\to\rn$ such that
\[\int_\Omega M(x,\xi(x))\,dx<\infty.\]
\item[ii)]${L}_M(\Omega;\rn)$  - the generalised Orlicz-Musielak space is the smallest linear space containing ${\cal L}_M(\Omega;\rn)$, equipped with the Luxemburg norm 
\[||\xi||_{L_M}=\inf\left\{\lambda>0:\int_\Omega M\left(x,\frac{\xi(x)}{\lambda}\right)\,dx\leq 1\right\}.\]
\item[iii)] ${E}_M(\Omega;\rn)$  - the closure in $L_M$-norm of the set of bounded functions.
\end{itemize}
\end{defi}
Then 
\[{E}_M(\Omega;\rn)\subset {\cal L}_M(\Omega;\rn)\subset { L}_M(\Omega;\rn),\]
the space ${E}_M(\Omega;\rn)$ is separable and $({E}_M(\Omega;\rn))^*=L_{M^*}(\Omega;\rn)$, see~\cite{gwiazda-non-newt,Aneta}.

Under the so-called $\Delta_2$-condition (Definition~\ref{def:D2}) we would be equipped with stronger tools. Indeed,  if $M\in\Delta_2$, then
\[{E}_M(\Omega;\rn)= {\cal L}_M(\Omega;\rn)= {L}_M(\Omega;\rn)\]
and $L_M(\Omega;\rn)$ is separable. When both  $M,M^*\in\Delta_2$, then $L_M(\Omega;\rn)$ is separable and reflexive, see~\cite{GMWK,gwiazda-non-newt}. We face the problem without this structure.

\begin{rem} Definition~\ref{def:Nf} (see points 3 and 4)  implies 
$\lim_{|\xi|\to \infty}\inf_{x\in\Omega}\frac{M^*(x,\xi)}{|\xi|}=\infty$ and 
  $\inf_{x\in\Omega}M^*(x,\xi)>0$ for any $\xi\neq 0$. Then, consequently, Lemma~\ref{lem:M*<M} ensures 
\begin{equation}
\label{LinfinLM}L^\infty(\Omega;\rn) \subset L_M(\Omega;\rn).\end{equation} 
\end{rem} 

\subsubsection*{Comments on assumptions on $A$ }

The following consideration explains how condition (A2) settles growth and coercivity condition on the leading part of the operator.

In the standard $L^p$-setting it is enough to note that (A2) implies directly \[A(x,\xi)\xi\geq c_A |\xi|^p\] and $|A(x,\xi)|\cdot|\xi|\geq \wt{c}_A |A(x,\xi)|^{p'},$ leading further to the condition \[\wt{c}_A |\xi|^{p-1} \geq  |A(x,\xi)|.\]

In the nonstandard growth setting, considering the first counterpart of the above condition, i.e.
\begin{equation}
\label{Adown}A(x,\xi)\xi\geq c_A M(x,\xi),
\end{equation} we get the minimal growth. As for the bound from above, we define an increasing function $P:\R\cup\{0\}\to\R\cup\{0\}$ by the following formula 
\[P(s):=\sup_{\xi:\ |\xi|=s}\left(\inf_{x\in\Omega, } M^*(x, \xi ) \right)^*.\] 
 Notice that for every $x\in\Omega$ and $\xi\in\rn$ such that $|\xi|=s$ it holds $ P(s) \geq M(x,\xi)$.  Moreover, we have an upper bound for the growth of the operator
\begin{equation}
\label{Aup}
 |A(x,\xi)| \leq 2 (P^*)^{-1}\left(\frac{1}{c_A}P\left(\frac{2}{c_A}|\xi |\right)\right) .
\end{equation} 
Indeed, to prove
\[c_A P^*\left(\frac{1}{2}|A(x,\xi)|\right) \leq  P\left(\frac{2}{c_A}|\xi | \right)\] it suffices to notice that
Fechel-Young inequality~\eqref{inq:F-Y} yields
\[A(x,\xi)\xi\leq P\left(\frac{2}{c_A}|\xi|\right)+P^*\left(\frac{c_A}{2}|A(x,\xi)|\right)\leq P\left(\frac{2}{c_A}|\xi|\right)+c_A P^*\left(\frac{1}{2}|A(x,\xi)|\right),\] 
whereas on the other hand
\[A(x,\xi)\xi\geq  {c_A} M^* (x,A(x,\xi))\geq c_A P^*\left(|A(x,\xi)|\right)\geq 2c_A P^*\left(\frac{1}{2}|A(x,\xi)|\right).\]
Conditions of this form are considered in classical Orlicz setting, when $M(x,\xi)=M(|\xi|)$ by e.g.~\cite{Gossez,Mustonen}. Note that then we can take $P(s)=M(s)$. Since (A2) implies~\eqref{Adown} and~\eqref{Aup}, we assume particular growth and coercivity of the leading part of the operator corresponding to the modular function of the space, where the solutions are defined. Nonetheless, conditions~\eqref{Adown} and~\eqref{Aup} are not sufficient in our approach. Note that they do not ensure that the operator and the solution are in the proper dual spaces. Let us stress further that the consequences of (A2) are  expressed by $N$-functions of general type of growth.

\subsubsection*{Main tools}

The existence of solutions to the truncated problem follows directly from~\cite[Theorem~1.5]{GMWK}.

\begin{theo}[Existence with bounded data, cf.~\cite{GMWK}]\label{theo:boundex} Suppose $g\in L^\infty(\Omega)$, an $N$-function $M$ satisfies assumption (M) and function $A$ satisfies assumptions (A1)-(A3). Then there exists a weak solution to the problem \[
\left\{\begin{split}
-\dv A(x,\nabla u)= g &\qquad \mathrm{ in}\qquad  \Omega,\\
u(x)=0 &\qquad \mathrm{  on}\qquad \partial\Omega,
\end{split}\right.
\] Namely, there exists $u \in W_0^{1,1} (\Omega)$ such that $\nabla u \in L_M (\Omega)$ satisfies
\[\int_\Omega A(x, \nabla u) \cdot \nabla \varphi dx =
\int_\Omega g\varphi\,dx,\]
for all $\varphi \in C_0^\infty(\Omega)$. Moreover, $A(\cdot,\nabla u)\in L_{M^∗}(\Omega)$.
\end{theo}
 
In fact, \cite[Theorem~1.5]{GMWK} is proven under the assumption that there exists $F:\Omega\to\rn$, such that $g= \dv F$ and $F\in E_{M^*}(\Omega)$. Nevertheless, each bounded $g$ is of this form. Existence of such $F$ is clear, while the fact that $F\in E_{M^*}(\Omega)$ is a consequence of properties of the Bogovski operator, see e.g.~[\cite{NavierStokes}, Lemma II.2.1.1]. 

\bigskip

The following  refined approximation result of~\cite[Theorem~2.7]{GMWK} being an~improvement of the case from~\cite{BenkiraneDouieb} is proven in Appendix. 
\begin{theo}[Approximation theorem]\label{theo:approx} 
Let $\Omega$ be a Lipschitz domain and an $N$-function $M$~satisfy condition (M). Then for any $\vp$ such that $\vp\in V_0^M\cap L^\infty(\Omega)$  there exists a sequence $\{\vp_\delta\}_{\delta>0}\in C_0^\infty(\Omega)$ converging modularly to $\vp$, i.e. such that $\nabla\vp_\delta\xrightarrow[]{M}\nabla \vp$.
\end{theo}
The vital tool in our study is the following modular Poincar\'{e}-type inequality. The proof is also included in Appendix.
\begin{theo}[Modular Poincar\'e inequality]\label{theo:Poincare} 
Let $m:\rp\to\rp$ be an arbitrary function satisfying $\Delta_2$-condition and $\Omega\subset\rn$ be a bounded domain,
then there exist $c=c(\Omega,N,m)>0$ such that for every $g\in W^{1,1}(\Omega)$, such that $\int_\Omega m(|\nabla g|)dx<\infty$, we have
\[\int_\Omega m(|g|)dx\leq c
\int_\Omega m(|\nabla g|)dx.\]

\end{theo}

\section{The main proof}

\begin{proof}[Proof of Theorem~\ref{theo:main}] The proof is divided into several steps.

\medskip

\textbf{Step 1. Truncated problem.} Existence to a truncated problem 
\begin{equation}
\label{prob:trunc}\left\{\begin{split}
-\dv A(x,\nabla u_s)= T_s(f) &\qquad \mathrm{ in}\qquad  \Omega,\\
u_s(x)=0 &\qquad \mathrm{  on}\qquad \partial\Omega,
\end{split}\right.
\end{equation}
for $s>0$ is a direct consequence of Theorem~\ref{theo:boundex} with $g=T_s(f)$ (truncation $T_s$ comes from~\eqref{Tk}).

\medskip

\textbf{Step 2. A priori estimates.} In order to obtain uniform integrability of sequences $\{A(x,\nabla T_k(u_s))\}_{s>0}$ and $\{\nabla T_k(u_s)\}_{s>0}$ we need to obtain the following a priori estimates.

For $u_s$ being a weak solution to~\eqref{prob:trunc}, $s>0$ and $f\in L^1(\Omega)$, we have the following estimates for any $k>0$
\begin{eqnarray}
\int_\Omega M(x, \nabla T_k(u_s))dx&\leq& c k \|f\|_{L^1(\Omega)},\label{Mapriori}\\
\int_\Omega M^*(x, A(x,\nabla T_k(u_s)))dx&\leq& c k \|f\|_{L^1(\Omega)},\label{M*apriori}
\end{eqnarray}
where the constant $c$ depends only on the growth condition~(A2).

Indeed, considering $(T_k(u_s))_\delta$ -- a sequence approximating $T_k(u_s)$ as in Theorem~\ref{theo:approx}, we get
\begin{multline*}
\int_\Omega A(x, \nabla T_k(u_s))\nabla T_k(u_s)dx=\lim_{\delta\to 0}\int_\Omega A(x, \nabla T_k(u_s))\nabla T_k(u_s)dx= \\
=\lim_{\delta\to 0}\int_\Omega T_s(f) ( T_k(u_s))_\delta dx=\int_\Omega T_s(f)  T_k(u_s)dx.
\end{multline*}
We observe that due to Assumption (A2) we have
\begin{multline*}\int_\Omega c_A\left(M(x,\nabla T_k(u_s))+M^*(x,A(x, \nabla T_k(u_s)))\right)dx\leq\\\leq \int_\Omega A(x,\nabla T_k(u_s))\nabla T_k(u_s)dx=\int_\Omega T_s(f)  T_k(u_s)dx\leq k \|f\|_{L^1(\Omega)}.\end{multline*}
Estimates \eqref{Mapriori} and \eqref{M*apriori} are direct consequences of the above one. Then, according to Lemma~\ref{lem:unif}, we reach the goal of this step. 

\medskip

\textbf{Step 3. Controlled radiation.} The proof of this step is a modification of~\cite[Lemma~5.1, Corollary~5.2]{gwiazda-ren-ell}. We consider the $N$-function $\dm:\r_+\cup\{0\}\to\r$ defined as follows. Let
\begin{equation}
\label{dm}
m_*(r)=\left(\inf_{x\in\Omega,\ |\xi|=r} M(x,\xi)\right)^{**}.
\end{equation} 
Then, let $\dm$ be a solution to the differential equation
\[\dm'(s)=\left\{\begin{array}{ll}
m'_*(s)&\text{for }s:\ m'_*(s)\leq \alpha\frac{m_*(s)}{s},\\
\alpha\frac{m_*(s)}{s}&\text{for }s:\ m'_*(s)> \alpha\frac{m_*(s)}{s},
\end{array}\right.\]
with the initial condition $\dm(0)=0=m_*(0)$ and a certain $\alpha>1$. Note that $\dm'(s)\leq m'_*(s)$  for every $s$, so $\dm(s)\leq m_*(s)$. Due to Lemma~\ref{rem:2ndconj} also $ m_*(s)\leq  \inf_{x\in\Omega,\ |\xi|=s} M(x,\xi)$ for every $s$. Thus 
\[\dm(s)\leq \inf_{x\in\Omega,\ |\xi|=s} M(x,\xi)\qquad\forall_{s\in\r_+\cup\{0\}}.\]
Moreover, by~\cite[Chapter~II.2.3, Theorem~3, point~1. (ii)]{rao-ren} $\dm$ satisfies $\Delta_2$-condition (cf.~\eqref{D2} without dependence on $x$). 

\begin{prop}
Suppose $u_s$ is a weak solution to~\eqref{prob:trunc}, $s>0$ and $f\in L^1(\Omega)$. Then  there exist   $c>0$ and $\gamma:\R_+\to\R_+$, such that for every $l>0$
\begin{equation}
\label{a<gamma}
\int_{\{l<|u_s|<l+1\}}A(x,\nabla u_s)\nabla u_s dx\leq  \gamma\left(\frac{l}{\dm(l)}\right),
\end{equation} 
and $\gamma$ is independent of $l,s$ and  $\lim_{r\to 0}\gamma(r)=0$.

\end{prop}
\begin{proof} 
Note that for $\dm$ given by~\eqref{dm} we have\[|\{|u_s|\geq l\}|=|\{|T_l(u_s)|= l\}|=|\{|T_l(u_s)|\geq l\}|=|\{\dm(|T_l(u_s)|)\geq \dm(l)\}|.\]
Moreover, for $l>0$ we have
\begin{equation*}
\begin{split}
|\{|u_s|\geq l\}|&\leq \int_{\Omega} \frac{\dm(|T_l(u_s)|)}{\dm(l)} dx\leq \frac{c(N,\Omega)}{\dm(l)} \int_\Omega\dm( |\nabla T_l(u_s)|)dx\leq \\
&\leq  \frac{c(N,\Omega)}{\dm(l)} \int_\Omega M(x,  \nabla T_l(u_s) )dx \leq 
 \frac{C(M,N,\Omega)}{\dm(l)}  \cdot l\|f\|_{L^1(\Omega)}\leq\\&\leq C(f,M,N,\Omega)\frac{l}{\dm(l)} 
.\end{split}\end{equation*}
In the above estimates we apply (respectively) the Chebyshev inequality, the Poincar\'{e} inequality (Theorem~\ref{theo:Poincare}), a priori estimate~\eqref{Mapriori} and the facts that $f\in L^1(\Omega)$ and that $\dm$ is an $N$-function (cf.~Definition~\ref{def:Nf}). Thus,  there exists $\gamma:\r_+\to\r_+$ independent of $l,s$, for which  $\lim_{r\to \infty}\gamma(r)=0$. Moreover, $\int_E|f|\, dx\leq \gamma(|E|).$ In particular,
\begin{equation}
\label{intfgammaul}
\int_{\{|u_s|\geq l\}} |f| dx \leq \gamma\left(\frac{l}{\dm(l)}\right).
\end{equation} 

As for the second assertion let us define $\psi_l:\R\to\R$ by 
\begin{equation}
\psi_l(r):=
\min\{(l+1-|r|)^+,1\}
\label{psil}\end{equation}
and consider $(\psi_l(u_s))_\delta$ -- a sequence approximating $\psi_l(u_s)$ as in Theorem~\ref{theo:approx}. Using $(1-\psi_l(u_s))_\delta$ as a test function in~\eqref{prob:trunc} we get
\begin{multline*}\int_\Omega A(x, \nabla u_s)\nabla (1-\psi_l(u_s))dx=\lim_{\delta\to 0} \int_\Omega A(x, \nabla u_s)\nabla (1-\psi_l(u_s))dx=\\
=\lim_{\delta\to 0} \int_\Omega T_s(f)(1-\psi_l(u_s))_\delta dx=
\int_\Omega T_s(f)(1-\psi_l(u_s))dx.\end{multline*}
We notice that the meaning of truncations and the form of $\psi_l$, together with~\eqref{intfgammaul} implies 
\[\begin{split}&\int_{\{l<|u_s|<l+1\}} A(x, \nabla u_s)\nabla u_s\,dx=\\
&=\int_{\{l<|u_s|<l+1\}} A(x, \nabla T_{l+1}(u_s))\nabla T_{l+1}(u_s)dx=\int_\Omega A(x, \nabla u_s)\nabla (1-\psi_l(u_s))dx=\\
&=\int_\Omega T_s(f)(1-\psi_l(u_s))dx\leq \int_{\{|u_s|\geq l\}}|f| dx\leq \gamma\left(\frac{l}{\dm(l)}\right),\end{split}\]
which was the aim.
\end{proof}

\medskip

\textbf{Step 4. Convergence of truncations}

\begin{prop}
Suppose an $N$-function $M$ satisfies assumption (M) and function $A$ satisfies assumptions (A1)-(A3). For $s>0$ and $f\in L^1(\Omega)$ let $u_s$ be a weak solution to~\eqref{prob:trunc}. Then there exists a measurable function $u:\Omega\to\r$, such that $T_k(u)\in V_0^M$, being a limit of some subsequence of $\{u_s\}_s$ in the following sense 
\begin{eqnarray}
&u_s\to u  &a.e.\ \text{in}\ \Omega,\label{conv:usae}\\
&|\{|u|>l\}|\leq  \gamma\left(\frac{l}{\dm(l)}\right),& l\in\N,\label{conv:umeas}
\end{eqnarray}
and for each $k\in\N$ and $s\to\infty$
\begin{eqnarray}
&T_k(u_s)\xrightarrow{} T_k(u)\quad \text{strongly\ in}\ L^p(\Omega)\ \text{for}\ p\in[1,\infty),&\label{conv:TuLp}\\
&\nabla T_k(u_s)\xrightharpoonup{} \nabla T_k(u)\quad \text{weakly\ in}\ L^1(\Omega),&\label{conv:nTuwL}\\
&\nabla T_k(u_s)\xrightharpoonup{*} \nabla T_k(u)\quad \text{weakly}-*\ \text{in}\ L_M(\Omega;\rn),&\label{conv:nTuwLM}\\
&A(x,\nabla T_k(u_s))\xrightharpoonup{*} A(x,\nabla T_k(u))\quad  \text{weakly}-*\ \text{in}\ L_{M^*}(\Omega;\rn).&\label{conv:ATuwLMs}
\end{eqnarray}
\end{prop}

\begin{proof}The proven a priori estimate~\eqref{Mapriori} \begin{eqnarray*}
\int_\Omega M(x, \nabla T_k(u_s))dx&\leq& c k \|f\|_{L^1(\Omega)} 
\end{eqnarray*}
implies that for each $k$ the sequence $(T_k(u_s))_{s=1}^{\infty}$ is bounded in $W^{1,1}_0(\Omega)$. Hence, there exists a function $u$ such that
\begin{eqnarray*}
T_k(u_s)&\xrightarrow[s\to\infty]{}& T_k(u)\ \text{strongly in } L^1(\Omega),\\
\nabla T_k(u_s)&\xrightharpoonup[s\to\infty]{ }& \nabla T_k(u)\ \text{weakly in } L^1(\Omega;\rn),\\
\nabla T_k(u_s)&\xrightharpoonup[s\to\infty]{*}& \nabla T_k(u)\ \text{weakly-$*$ in } L_M(\Omega;\rn),
\end{eqnarray*}
in particular implying~\eqref{conv:nTuwL} and~\eqref{conv:nTuwLM}. Furthermore, the Lebesgue Monotone Convergence Theorem implies
\[
 u_s \xrightarrow[s\to\infty]{}  u \quad \text{strongly in } L^1(\Omega),\]
and  up to a subsequence we have~\eqref{conv:usae}, i.e.
\[
 u_s \xrightarrow[s\to\infty]{}  u \quad \text{a.e. in } \Omega.\]

Since $\Omega$ is bounded, for  fixed $k\in\N$ convergence in~\eqref{conv:TuLp} results from uniform integrability in~$L^p(\Omega)$ of bounded functions $T_k(u_s)$ combined with the Vitali Convergence Theorem (Theorem~\ref{theo:VitConv}). Meanwhile, the Dominated Convergence Theorem (due to~\eqref{intfgammaul}) gives~\eqref{conv:umeas}.

On the other hand, if for every $k$ we denote
\[\ask=A(x,\nabla T_{k}(u_s(x))),\]
then it follows from~\eqref{M*apriori} that there exists  $\cA_k\in L_{M^*}(\Omega;\rn)$ such that
\begin{equation}
\label{a-conv-ca}\ask\xrightharpoonup{*} \cA_k \quad  \text{weakly}-*\ \text{in}\ L_{M^*}(\Omega;\rn).
\end{equation}
Our aim is now to show that in~\eqref{a-conv-ca}
\begin{equation}
\label{lim=ca}
\cA_k(x)=A(x,\nabla T_k(u )).
\end{equation}

We take approximating sequence of smooth functions $\nabla (T_k(u))_{\delta}\xrightarrow[\delta\to 0]{M} \nabla T_k(u)$ (cf. Theorem~\ref{theo:approx}) and show that\begin{equation}
\label{limsup2}
\lim_{l\to\infty}
\lim_{\delta\to 0}\limsup_{s\to\infty}
\int_\Omega \asl \psi_l(u_s)\nabla \left[T_k(u_s)-(T_k(u))_\delta\right]dx= 0.
\end{equation}
Testing~\eqref{prob:trunc} by $\vp=\psi_l(u_s)(T_k(u_s)-(T_k(u))_\delta),$ where $\psi_l$ is given by~\eqref{psil},  we get
\begin{equation}
\label{cf.C5}
\int_\Omega A(x,\nabla u_s)\nabla \left[ \psi_l(u_s)(T_k(u_s)-(T_k(u))_\delta)\right]dx=\int_\Omega T_s(f)\psi_l(u_s)(T_k(u_s)-(T_k(u))_\delta)dx.
\end{equation}
We observe that the right-hand side of~\eqref{cf.C5} tends to zero, i.e.
\[\lim_{l\to\infty}\lim_{\delta\to 0}\lim_{s\to\infty} \int_\Omega T_s(f)\psi_l(u_s)(T_k(u_s)-(T_k(u))_\delta)dx=0.\]
Indeed, the convergence a.e. is ensured by~\eqref{conv:usae} and to apply the Lebesgue Dominated Convergence Theorem we note
\[\begin{split}&\lim_{\delta\to 0}\lim_{s\to\infty}\left|\int_\Omega T_s(f)\psi_l(u_s)(T_k(u_s)-(T_k(u))_\delta)dx\right|\leq\\
& \leq \lim_{\delta\to 0}\lim_{s\to\infty}\int_\Omega |T_s(f)|\psi_l(u_s)\cdot|T_k(u_s)-T_k(u)|dx+\lim_{\delta\to 0}\lim_{s\to\infty}\int_\Omega |T_s(f)|\psi_l(u_s)\cdot|T_k(u)-(T_k(u))_\delta|dx\leq \\ 
&\leq \lim_{\delta\to 0}\lim_{s\to\infty}\int_\Omega |f|\cdot 2 k\,dx+\lim_{\delta\to 0}\lim_{s\to\infty}\int_\Omega |f|\cdot|T_k(u)-(T_k(u))_\delta|dx=\\
&=2k\|f\|_{L^1(\Omega)}+\lim_{\delta\to 0} \int_\Omega |f|\cdot|T_k(u)-(T_k(u))_\delta|dx.\end{split}
\] The last expression is convergent due to Lemma~\ref{lem:TM1}.

Let us now concentrate on the left-hand side of~\eqref{cf.C5}:
\begin{equation*}
\begin{split}
&\int_\Omega A(x,\nabla u_s)\nabla \left[ \psi_l(u_s)(T_k(u_s)-(T_k(u))_\delta)\right]dx=\\
=&\int_\Omega A(x,\nabla u_s)\nabla  \psi_l(u_s)\left[T_k(u_s)-(T_k(u))_\delta\right]dx+\int_\Omega A(x,\nabla u_s) \psi_l(u_s)\nabla \left[T_k(u_s)-(T_k(u))_\delta\right]dx=\\
=&I_1+I_2,\end{split}
\end{equation*}
where due to~\eqref{a<gamma} we have
\[\begin{split}
\lim_{l\to\infty}\left(\lim_{\delta\to 0} \limsup_{s\to\infty} |I_1|\right)&\leq\lim_{l\to\infty}\left(\lim_{\delta\to 0} \limsup_{s\to\infty}\int_{\{l<|u_s|<l+1\}} A(x,\nabla u_s)\nabla u_s|T_k(u_s)-T_k(u) |dx\right)+\\&+\lim_{l\to\infty}\left(\lim_{\delta\to 0} \limsup_{s\to\infty}\int_{\{l<|u_s|<l+1\}} A(x,\nabla u_s)\nabla u_s|T_k(u)-(T_k(u))_\delta|dx\right)=\\
&= II_1+II_2.\end{split}\] 
Moreover,
\[\begin{split}II_1&\leq \lim_{l\to\infty}\left(2k\limsup_{s\to\infty}\int_{\{l<|u_s|<l+1\}} A(x,\nabla u_s)\nabla u_sdx\right)\leq \lim_{l\to\infty}\left[2k\gamma\left(\frac{l}{\dm(l)}\right)\right]=0,\end{split}\] 
meanwhile the convergence of $II_2$ results from Lemma~\ref{lem:TM1}.

Then passing to the limit in~\eqref{cf.C5} we obtain\begin{equation}
\label{limsup}
\lim_{l\to\infty}\lim_{\delta\to 0}\limsup_{s\to\infty} I_2
=\lim_{l\to\infty}\lim_{\delta\to 0}\limsup_{s\to\infty}
\int_\Omega A(x,\nabla u_s) \psi_l(u_s)\nabla \left[T_k(u_s)-(T_k(u))_\delta\right]dx= 0.
\end{equation}
Then~\eqref{limsup} is equivalent to~\eqref{limsup2}.

\medskip

Before we apply monotonicity trick, we need to show that
\begin{equation}
\label{limtrunc}
\lim_{l\to\infty}\lim_{\delta\to 0}\limsup_{s\to\infty}
\int_\Omega \ask \psi_l(u_s)\nabla \left[T_k(u_s)-(T_k(u))_\delta\right]dx= 0.
\end{equation}

Taking into account~\eqref{limsup2}, the equality~\eqref{limtrunc} will be proven when the following expression is shown to tend to $0$ (still $k\leq l$)\begin{equation}
\begin{split}
&III=\int_\Omega (\ask -\asl) \psi_l(u_s)\nabla \left[T_k(u_s)-(T_k(u))_\delta\right]dx=\\
=&\int_{\Omega} (\asl -A(x,0)) \mathds{1}_{\{k<|u_s| \}}\psi_l(u_s) \nabla (T_k(u))_\delta dx=\\
=&\int_{\Omega} \asl  \mathds{1}_{\{k<|u_s| \}} \psi_l(u_s) \nabla (T_k(u))_\delta dx.
\end{split}
\end{equation}
We prove that
\begin{equation}\label{asl<al}
\begin{split}
\lim_{\delta\to 0}\limsup_{s\to\infty} |III|&\leq \lim_{\delta\to 0}\limsup_{s\to\infty}\int_{\Omega} |\asl|  \mathds{1}_{\{k<|u_s| \}} \psi_l(u_s) |\nabla (T_k(u))_\delta|\, dx\leq \\
& \leq \lim_{\delta\to 0} \int_{\Omega} |{\cal A}_{l+1}|  \mathds{1}_{\{k<|u| \}} \psi_l(u) |\nabla (T_k(u))_\delta|\, dx.
\end{split}
\end{equation}
For this we will use Lemma~\ref{lem:TM1} with
\[w^s=|\asl|\cdot|\nabla (T_k(u))_\delta|\xrightharpoonup[s\to\infty]{L^1(\Omega)}|{\cal A}_{l+1}|\cdot|\nabla (T_k(u))_\delta|=w.\]
The convergence $w^s\xrightharpoonup{} w$ is a consequence of~\eqref{a-conv-ca}.  Let 
$v^s=\mathds{1}_{\{k<|u_s| \}}\psi_l(u_s)$ and $v^s_\ve\in C(\Omega)\cap L^\infty (\Omega)$ with $\ve\geq 0$ be given by
\[v^s_\ve=\left\{\begin{array}{ll}
1 & k<|u_s|<l,\\
\text{affine} & k-\ve\leq |u_s|\leq k,\ l\leq |u_s|\leq l+1\\
0 & |u_s|<k-\ve,\ |u_s|>l+1.\\
\end{array}\right.\]
Notice that for $s\to\infty$  and every $\ve>0$, due to continuity of $v_\ve^s$, we have
\[v^s_\ve\xrightarrow{a.e.}v_\ve:=\left\{\begin{array}{ll}
1 & k<|u|<l,\\
\text{affine} & k-\ve\leq |u|\leq k,\ l\leq |u|\leq l+1\\
0 & |u|<k-\ve,\ |u|>l+1.\\
\end{array}\right.\]

Furthermore, for every $s$ we have
\begin{equation}
\label{g-ep-s}
\int_\Omega |\asl|\cdot|\nabla (T_k(u))_\delta|v^s\,dx\leq\int_\Omega |\asl|\cdot|\nabla (T_k(u))_\delta|v^s_\ve\,dx.
\end{equation}

Since $v_\ve\in L^\infty (\Omega)$, Lemma~\ref{lem:TM1} yields $\int_\Omega w^s v^s_\ve\, dx\to \int_\Omega w v_\ve\, dx,$ that is
\[\lim_{s\to\infty}\int_\Omega |\asl|\cdot|\nabla (T_k(u))_\delta|v^s_\ve\,dx=\int_\Omega |\al|\cdot|\nabla (T_k(u))_\delta|v_\ve\,dx.\]
The Lebesgue Monotone Convergence Theorem implies
\begin{equation}
\label{g-ep-ve}\lim_{\ve\to 0} \int_\Omega |\al|\cdot|\nabla (T_k(u))_\delta|v_\ve\,dx=\int_\Omega |\al|\cdot|\nabla (T_k(u))_\delta|v_0\,dx=\int_\Omega |\al|\cdot|\nabla (T_k(u))_\delta|\mathds{1}_{\{k<|u| \}}\psi_l(u)\,dx.
\end{equation}

Thus~\eqref{g-ep-s} together with~\eqref{g-ep-ve} give
\[\begin{split}
\limsup_{s\to \infty} \int_\Omega |\asl|\cdot|\nabla (T_k(u))_\delta|g^s\,dx\leq\int_\Omega |\al|\cdot|\nabla (T_k(u))_\delta|\mathds{1}_{\{k<|u| \}}\psi_l(u)\,dx\end{split}\]
and we get~\eqref{asl<al}.

Our aim now is to prove
\begin{equation}
\label{al-delta}\lim_{\delta\to 0}\int_\Omega |\al|\cdot|\nabla (T_k(u))_\delta|\mathds{1}_{\{k<|u| \}}\psi_l(u)\,dx=0
\end{equation}
Recall that $\nabla (T_k (u))_\delta\xrightarrow{M}\nabla T_k (u)$. Therefore by Definition~\ref{def:convmod} ii), the sequence $\{M(x,\nabla(T_k (u))_\delta/\lambda)\}_\delta$ is uniformly bounded in $L^1(\Omega;\rn)$ for some $\lambda$ and consequently, by~Lemma~\ref{lem:unif}  $\{\nabla(T_k (u))_\delta\}_\delta$ is uniformly integrable. Hence the Vitali Convergence Theorem (Theorem~\ref{theo:VitConv}) gives 
\[\lim_{\delta\to 0}\int_\Omega |\al|\cdot|\nabla (T_k(u))_\delta|\mathds{1}_{\{k<|u| \}}\psi_l(u)\,dx=\int_\Omega |\al|\cdot|\nabla T_k(u) |\mathds{1}_{\{k<|u| \}}\psi_l(u)\,dx,\]
which is equal to zero, because $T_k(u) |\mathds{1}_{\{k<|u| \}}=0$. Thus~\eqref{al-delta} and~\eqref{limtrunc} hold.
\medskip
  
We observe that we can remove $\psi_l(u_s)$ from~\eqref{limtrunc}. Indeed, notice that for $l\geq k$ due to~Lemma~\ref{lem:M*<M} we have 
\begin{equation*}
\begin{split}
&\int_\Omega \ask \psi_l(u_s)\nabla \left[T_k(u_s)-(T_k(u))_\delta\right]dx=\\ 
&=\int_\Omega \ask \nabla \left[T_k(u_s)-(T_k(u))_\delta\right]dx-\int_{\{|u_s|>l\}} A(x,0) (\psi_l(u_s)-1)\nabla (T_k(u))_\delta dx=\\ 
&=\int_\Omega \ask \nabla \left[T_k(u_s)-(T_k(u))_\delta\right]dx.\end{split}
\end{equation*}
Therefore,~\eqref{limtrunc} is equivalent to 
\begin{equation}\label{lim<<}
\lim_{\delta\to 0}\limsup_{s\to\infty}\int_\Omega \ask\nabla \left[T_k(u_s)-(T_k(u))_\delta\right]dx= 0.
\end{equation}

Now we apply the Minty-Browder monotonicity trick. Since~\eqref{a-conv-ca}, then for each $\delta$
\begin{equation}
\label{lim<A}\lim_{s\to\infty} 
\int_\Omega \ask \nabla (T_k(u))_\delta dx=  \int_\Omega {\cal A}_k\cdot \nabla (T_k(u))_\delta dx.
\end{equation}
Then~\eqref{lim<<} together with~\eqref{lim<A} imply
\begin{equation}
\label{lim<A'} \limsup_{s\to\infty} 
\int_\Omega \ask \nabla T_k(u_s) dx=\lim_{\delta\to 0} \int_\Omega {\cal A}_k\cdot \nabla (T_k(u))_\delta dx=\int_\Omega {\cal A}_k\cdot \nabla T_k(u) dx,
\end{equation}
where the last equality is obtained analogically as~\eqref{al-delta}.

Monotonicity of~$A$ results in
\[\int_\Omega \ask \nabla T_k(u_s) dx\geq \int_\Omega \ask \eta\ dx+\int_\Omega A(x,\eta)  (\nabla T_k(u_s)-\eta)\ dx\]
for any $\eta\in \rn$.  Taking upper limit with ${s\to\infty}$ above (due to~\eqref{lim<A'},~\eqref{a-conv-ca}, and~\eqref{conv:nTuwLM}) we get
\begin{equation*}
\int_\Omega {\cal A}_k\cdot \nabla T_k(u) dx\geq \int_\Omega {\cal A}_k\cdot \eta\ dx+\int_\Omega A(x,\eta)  (\nabla T_k(u)-\eta)\ dx.\end{equation*} 
Note that it is equivalent to
\begin{equation}
\label{Ak-mon} \int_\Omega( {\cal A}_k- A(x,\eta) )( \nabla T_k(u)-\eta) dx\geq  0.
\end{equation} 

Let us define\begin{equation}
\label{omm}
\Omega_K=\{x\in\Omega:\ |\nabla T_k(u)|\leq K\}.
\end{equation} 
Then, in~\eqref{Ak-mon} we choose
\[\eta=\nabla T_k(u)\mathds{1}_{\Omega_i}+hz\mathds{1}_{\Omega_j},\]
where  $0<j<i$, $h\in\rp$ and $z\in L^\infty(\Omega;\rn)$, to get 
\begin{equation*}
\int_\Omega( {\cal A}_k- A(x,\nabla T_k(u)\mathds{1}_{\Omega_i}+hz\mathds{1}_{\Omega_j}) )( \nabla T_k(u)-\nabla T_k(u)\mathds{1}_{\Omega_i}-hz\mathds{1}_{\Omega_j}) dx\geq  0.
\end{equation*} 
Notice that it is equivalent to
\begin{equation}
\label{Ak-po-mon} \int_{\Omega\setminus\Omega_i}  {\cal A}_k\nabla T_k(u)dx- \int_{\Omega\setminus\Omega_i}A(x,0)\nabla T_k(u)dx +h\int_{ \Omega_j} (A(x,\nabla T_k(u)+hz)-{\cal A}_k)z  dx\geq  0.
\end{equation} 
The first and the second expression above tend to zero when $i\to\infty.$ Indeed, since ${\cal A}_k,A(x,0)\in L_{M^*}(\Omega;\rn)$ and $\nabla T_k(u)\in L_M(\Omega;\rn)$, the H\"older inequality~\eqref{inq:Holder} gives boudedness of integrands in $L^1(\Omega)$. Then we take into account shrinking domain of integration to get the desired convergence to $0$. In~particular, we can erase these expressions in~\eqref{Ak-po-mon} and divide the remaining expression by $h>0$, to obtain
\begin{equation*} \int_{ \Omega_j} (A(x,\nabla T_k(u)+hz)-{\cal A}_k)z  dx\geq  0.
\end{equation*}

Note that
\[A(x,\nabla T_k(u)+hz)\xrightarrow[h\to 0]{}A(x,\nabla T_k(u))\quad\text{a.e. in}\quad \Omega_j.\] 
Moreover, as $A(x,\nabla T_k(u)+hz)$ is bounded on $\Omega_j$, Lemma~\ref{lem:M*<M}  results in
\[\int_{ \Omega_j} M^*\left(x,A(x,\nabla T_k(u)+hz)\right)dx\leq \frac{2}{c_A} \sup_{h\in(0,1)} \int_{ \Omega_j} M\left(x,\frac{2}{c_A}A(x,\nabla T_k(u)+hz)\right)dx.\]
The right-hand side is bounded, because $(\nabla T_k(u)+hz)_h$ is uniformly bounded in $L^\infty(\Omega_j;\rn)\subset L_M(\Omega;\rn)$ (cf.~\eqref{omm} and~\eqref{LinfinLM}). Hence, Lemma~\ref{lem:unif} gives uniform integrability of $\left(A(x,\nabla T_k(u)+hz)\right)_h$. When we notice that $|\Omega_j|<\infty$, we can apply the Vitali Convergence Theorem (Theorem~\ref{theo:VitConv}) to get
\[A(x,\nabla T_k(u)+hz)\xrightarrow[h\to 0]{}A(x,\nabla T_k(u))\quad\text{in}\quad L^1(\Omega_j;\rn).\] 
Thus
\begin{equation*} \int_{ \Omega_j} (A(x,\nabla T_k(u)+hz)-{\cal A}_k)z  dx\xrightarrow[h\to 0]{} \int_{ \Omega_j} (A(x,\nabla T_k(u))-{\cal A}_k)z  dx.
\end{equation*}
Consequently,
\begin{equation*}  \int_{ \Omega_j} (A(x,\nabla T_k(u))-{\cal A}_k)z  dx\geq 0,
\end{equation*}
for any $z\in L^\infty(\Omega;\rn)$. Let us take 
\[z=\left\{\begin{array}{ll}-\frac{A(x,\nabla T_k(u))-{\cal A}_k}{|A(x,\nabla T_k(u))-{\cal A}_k|}&\ \text{if}\quad A(x,\nabla T_k(u))-{\cal A}_k\neq 0,\\
0&\ \text{if}\quad A(x,\nabla T_k(u))-{\cal A}_k\neq 0.
\end{array}\right.\]
We obtain 
\begin{equation*}  \int_{ \Omega_j} |A(x,\nabla T_k(u))-{\cal A}_k| dx\leq 0,
\end{equation*}
hence \[A(x,\nabla T_k(u))={\cal A}_k\qquad \text{a.e.}\quad\text{in}\quad \Omega_j.\]
Since $j$ is arbitrary, we have the equality a.e. in $\Omega$ and~\eqref{lim=ca} is satisfied.
\end{proof}

\textbf{Step 5. Renormalized solutions.}

We aim at proving that $u$ is a renormalized solution (see Introduction). At first we observe that $u$ satisfies (R1) and concentrate on (R2).

Since $T_k(u)\in V_0^M\cap L^\infty(\Omega)$, Theorem~\ref{theo:approx} ensures that there exists a sequence $\{u_r\}_r\subset C_0^\infty(\Omega)$ indexed with $r\to \infty$,  such that
\begin{eqnarray*}
&u_r\xrightarrow{} u\quad \text{a.e.\ in}\ \Omega,&\\
&\nabla T_k(u_r)\xrightharpoonup{*} \nabla T_k(u)\quad \text{weakly}-*\ \text{in}\ L_M(\Omega;\rn),&\\
&\nabla h(u_r)\xrightharpoonup{*} \nabla h(u)\quad  \text{weakly}\ \text{in}\ L_{M}(\Omega),&
\end{eqnarray*}
where $h\in C_c^1(\R)$ is arbitrary. 

We test~\eqref{prob:trunc} by $\psi_l(u_s)h(u_r)\phi$ with $\phi\in W^{1,\infty}_0(\Omega)$ and get
\[L_{s,r,l}=\int_\Omega A(x,\nabla u_s) \nabla [\psi_l(u_s)h(u_r)\phi]dx= \int_\Omega T_s(f)  \psi_l(u_s)h(u_r)\phi\,dx=R_{s,r,l}.\]

We notice at first that due to the Lebesgue Dominated Convergence Theorem it holds that
\[\lim_{l\to\infty}\lim_{r\to \infty}\limsup_{s\to\infty} R_{s,r,l}=\int_\Omega f h(u)\phi dx. \]
Meanwhile on the left-hand side\[L_{s,r,l}=\int_\Omega A(x,\nabla u_s) \nabla  \psi_l(u_s)h(u_r)\phi dx+\int_\Omega A(x,\nabla u_s) \psi_l(u_s)\nabla [h(u_r)\phi]dx=L^1_{s,r,l}+L^2_{s,r,l},\]
where
\[\lim_{l\to\infty}\lim_{r\to \infty}\limsup_{s\to\infty}|L^1_{s,r,l}|\leq \|h\|_{L^\infty(\Omega)}\|\phi\|_{L^\infty(\Omega)}\lim_{l\to\infty}\lim_{r\to \infty} \left(\sup_{s}\int_{\{l<|u_s|<l+1\}} A_{s,l}(x) \nabla T_{l+1} (u_s) dx\right)=0\]
 due to~\eqref{a<gamma}. As for $L^2_{s,r,l}$ we notice that when $s\to\infty$, up to a subsequence,
\begin{equation*}\asl\xrightharpoonup{ } A(x,\nabla T_{l+1}(u))\quad  \text{weakly} \ \text{in}\ L^{1}(\Omega).
\end{equation*}
Indeed, a priori estimate~\eqref{M*apriori} combined with Lemma~\ref{lem:unif} give uniform integrability. Then, taking into account weak-* convergence~\eqref{conv:ATuwLMs}, the Dunford-Pettis Theorem (Theorem~\ref{theo:dunf-pet}) ensures weak $L^1$-convergence up to a subsequence.

Moreover, note that\begin{eqnarray*}
&|\psi_l(u_s)|\leq 1,&\\
&\nabla (h(u_r)\phi)\in L^\infty(\Omega;\rn).&
\end{eqnarray*} and for $s\to\infty$
\[\psi_l(u_s)\xrightarrow{} \psi_l(u)\quad \text{a.e.\ in}\ \Omega.\]
The sequence $\{  A(x,\nabla u_s) \psi_l(u_s)\nabla [h(u_r)\phi] \}_s$ is uniformly integrable. 
Due to the consequence of Chacon's Biting Lemma, Theorem~\ref{theo:bitinglemma}, we notice that \[\limsup_{r\to\infty}\limsup_{s\to\infty}\int_\Omega \asl\nabla h(u_r)\psi_l(u_s)dx=\int_\Omega A(x,\nabla T_{l+1}(u )) \psi_l(u )\nabla [h(u )\phi]dx.\]

Since $\mathrm{supp}\, h(u)\subset[-m,m]$ for some $m\in\N$ and we can consider only $l>m+1$. Then
\[\lim_{l\to\infty}\limsup_{r\to\infty}\limsup_{s\to\infty}L^2_{s,r,l}
=\lim_{l\to\infty}\int_\Omega A(x,\nabla T_{l+1}(u )) \psi_l(u )\nabla [h(u )\phi]dx=\int_\Omega A(x,\nabla  u  )  \nabla [h(u )\phi]dx.\]
and our solution $u$ satisfies condition (R2).

\medskip

Let us consider radiation control condition (R3), i.e.
\[\int_{\{l<|u|<l+1\}}A(x,\nabla u)\cdot\nabla u\, dx=\int_{\{l<|u|<l+1\}}A(x,\nabla T_{l+1}( u))\cdot\nabla  T_{l+1}(u )\, dx\xrightarrow[l\to\infty]{} 0 .\]
We follow the ideas of~\cite{gwiazda-ren-para} involving the Chacon Biting Lemma and the Young measure approach to show that for $s\to\infty$ it holds that
\begin{equation}
\label{117} \asl\cdot\nabla T_{l+1}(u_s)\xrightharpoonup{} A(x,\nabla T_{l+1}(u) )\cdot \nabla T_{l+1}(u)\qquad \text{weakly in } L^1(\Omega).
\end{equation}
First we observe that  the sequence $\{[\asl-A(x,\nabla T_{l+1}(u ))]\cdot [ \nabla T_{l+1}(u_s)-\nabla T_{l+1}(u )]\}_s$ is uniformly bounded in $L^1(\Omega)$. Indeed,
\[\begin{split}&\int_\Omega[\asl-A(x,\nabla T_{l+1}(u ))]\cdot [ \nabla T_{l+1}(u_s)-\nabla T_{l+1}(u )]dx\leq\\ 
&\qquad\leq \int_\Omega \asl \nabla T_{l+1}(u_s) dx
+\int_\Omega \asl \nabla T_{l+1}(u )  dx+\\
&\qquad+\int_\Omega A(x,\nabla T_{l+1}(u )) \nabla T_{l+1}(u_s) dx+\int_\Omega A(x,\nabla T_{l+1}(u )) \nabla T_{l+1}(u ) dx=IV_1+IV_2+IV_3+IV_4,
\end{split}\]
where $IV_1$ is uniformly bounded due to~\eqref{a<gamma} and $IV_4$ is independent of $s$. As for $IV_2$ we note 
\[\begin{split}
\limsup_{s\to\infty}IV_2&\leq\lim_{\delta\to 0}\limsup_{s\to\infty}\int_\Omega \asl (\nabla T_{l+1}(u)- \nabla( T_{l+1}(u))_\delta) dx +\lim_{\delta\to 0}\limsup_{s\to\infty}\int_\Omega \asl \nabla (T_{l+1}(u) )_\delta dx\leq\\
&= 0+\lim_{\delta\to 0} \int_\Omega \al \nabla (T_{l+1}(u) )_\delta dx=\int_\Omega \al \nabla T_{l+1}(u)   dx,\end{split}\]
where we applied~\eqref{lim<<},~\eqref{conv:ATuwLMs}, and then~\eqref{lim<A'}. Moreover, in the case of $IV_3$ the Fenchel-Young inequality and~\eqref{Mapriori} gives boundedness.

Then monotonicity of $A(x,\cdot)$ and  Theorem~\ref{theo:bitinglemma}  give, up to~a~subsequence, convergence
\begin{equation}
\label{110}
\begin{split}0&\leq [\asl-A(x,\nabla T_{l+1}(u ))]\cdot [\nabla T_{l+1}(u_s)-\nabla T_{l+1}(u )]\\
&\xrightarrow{b}\int_\rn [A(x,\lambda)-A(x,\nabla T_{l+1}(u ))]\cdot [\lambda-\nabla T_{l+1}(u )]d\nu_{x}(\lambda),\end{split}
\end{equation}
where $\nu_{x}$ denotes the Young measure generated by the sequence $\{\nabla T_{l+1}(u_s)\}_s$.

Since $\nabla T_{l+1}(u_s)\xrightharpoonup{}\nabla T_{l+1}(u )$ in $L^1(\Omega)$, we have $\int_\rn\lambda \, d\nu_x(\lambda)=\nabla T_{l+1}(u)$ for a.e. $x\in\Omega$. Then
\[\int_\rn  \asl \cdot [\lambda-\nabla T_{l+1}(u )]d\nu_{x}(\lambda)=0\]
and the limit in~\eqref{110} is equal for a.e. $x\in\Omega$ to
\begin{equation}
\label{111} \int_\rn [A(x,\lambda)-A(x,\nabla T_{l+1}(u ))]\cdot [\lambda-\nabla T_{l+1}(u )]d\nu_{x}(\lambda)=
\int_\rn  A(x,\lambda) \cdot  \lambda\, d\nu_{x}(\lambda)-\int_\rn  A(x,\lambda) \cdot \nabla  T_{l+1}(u )d\nu_{x}(\lambda).
\end{equation}

Uniform boundedness of the sequence $\{ \asl\nabla T_{l+1}(u_s) \}_s$ in $L^1(\Omega)$ (cf.~\eqref{a<gamma}) enables us  to apply once again Theorem~\ref{theo:bitinglemma} to obtain
\begin{equation*}
 \asl\nabla T_{l+1}(u_s)\xrightarrow{b} \int_\rn  A(x,\lambda) \cdot  \lambda\,d\nu_{x}(\lambda).
\end{equation*}
Moreover, assumption (A2) implies $\asl\nabla T_{l+1}(u_s)\geq 0$. Therefore, due to~\eqref{111} and~\eqref{110}, we have
\[\limsup_{s\to\infty} A(x,\nabla T_{l+1}(u_s) ) \nabla T_{l+1}(u_s)\geq \int_\rn  A(x,\lambda) \cdot  \lambda\,d\nu_{x}(\lambda).\]
Taking into account that in~\eqref{lim<A'} we can put ${\cal A}_k=A(x,\nabla T_{l+1}(u) )=\int_\rn  A(x,\lambda) \,d\nu_{x}(\lambda)$, the above expression implies\[\nabla T_{l+1}(u) \int_\rn  A(x,\lambda) \,d\nu_{x}(\lambda)\geq \int_\rn  A(x,\lambda) \cdot  \lambda\,d\nu_{x}(\lambda).\]

When we apply it, together with~\eqref{111}, the limit in~\eqref{110} is non-positive. Hence,
\begin{equation*} [\asl-A(x,\nabla T_{l+1}(u ))]\cdot [\nabla T_{l+1}(u_s)-\nabla T_{l+1}(u )]\xrightarrow{b} 0.
\end{equation*}
Observe further that $A(x,\nabla T_{l+1}(u ))\in L_{M^*}(\Omega;\rn)$ and we can choose ascending family of sets $E^{l+1}_j$, such that $| E^{l+1}_j|\to 0$ for $j\to \infty$ and $A(x,\nabla T_{l+1}(u ))\in L^\infty(\Omega\setminus E^{l+1}_j).$ Then, since $\nabla T_{l+1}(u_s)\xrightharpoonup{}\nabla T_{l+1}(u )$, we get\begin{equation*} A(x,\nabla T_{l+1}(u )) \cdot [\nabla T_{l+1}(u_s)-\nabla T_{l+1}(u )]\xrightarrow{b}  0
\end{equation*}
and similarly we conclude
\begin{equation*} \asl\cdot\nabla T_{l+1}(u )\xrightarrow{b}   A(x,\nabla T_{l+1}(u))\cdot\nabla T_{l+1}(u ).
\end{equation*}
Summing it up we get
\begin{equation*}
\asl\cdot\nabla T_{l+1}(u_s )\xrightarrow{b}  A(x,\nabla T_{l+1}(u))\cdot\nabla T_{l+1}(u ).
\end{equation*}
Recall that Theorem~\ref{theo:bitinglemma} together with~\eqref{lim<A'} and~\eqref{conv:ATuwLMs} results in~\eqref{117}. 

We turn back to prove (R3). Note that $\nabla u_s=0$ a.e. in $\{x\in\Omega:|u_s|\in\{l,l+1\}\}$. Then~\eqref{a<gamma} implies
\begin{equation*}
 \lim_{l\to\infty}\sup_{s>0}\int_{\{l-1<|u_s|<l+2\}}A(x,\nabla u_s)\cdot\nabla u_s\,dx=0.
\end{equation*}
For $g_l:\r\to\r$ defined by
\[g_l(r)=\left\{\begin{array}{ll}1&\text{if }\ l\leq |r|\leq l+1,\\
0&\text{if }\ |r|<l-1\text{ or } |r|> l+2,\\
\text{is affine} &\text{otherwise},
\end{array}\right.\]
we have
\begin{equation}
\label{128}
\int_{\{l-1<|u|<l+2\}}A(x,\nabla u)\cdot\nabla u\,dx\leq \int_{\Omega}g_l(u)A(x,\nabla T_{l+2}( u))\cdot\nabla  T_{l+2}( u)\,dx.
\end{equation}
Let us remind that we know that $u_s\to u$ a.e. in $\Omega$ (cf.~\eqref{conv:usae}) and $|\{x:|u_s|>l\}|\leq\gamma\left(l/\dm(l)\right)$ (cf.~\eqref{conv:umeas}). Moreover, we have weak convergence~\eqref{117}, $A(x,\nabla T_{l+2}( u_s))\cdot\nabla  T_{l+2}( u_s)>0$ and function $g_l$ is continuous and bounded. Thus, we infer that we can estimate the limit of the right-hand side of~\eqref{128} in the following way
\[\begin{split}
0&\leq \lim_{l\to\infty} \int_{\{l-1<|u|<l+2\}}A(x,\nabla u)\cdot\nabla u\,dx\leq \lim_{l\to\infty}\int_{\Omega}g_l(u)A(x,\nabla T_{l+2}( u))\cdot\nabla  T_{l+2}( u)\,dx=\\
&= \lim_{l\to\infty} \lim_{s\to\infty} \int_{\Omega}g_l(u)A(x,\nabla T_{l+2}(u_s))\cdot\nabla T_{l+2}(u_s)\,dx\leq\\
&\leq \lim_{l\to\infty}  \lim_{s\to\infty}\int_{\{l-1<|u|<l+2\}}A(x,\nabla T_{l+2}(u_s))\cdot\nabla T_{l+2}(u_s)\,dx=0,
\end{split}\]
where the last equality comes from~\eqref{a<gamma}.

Hence, our solution $u$ satisfies condition (R3) and  is a renormalized solution. 
\end{proof}

\section*{Appendix A}

\begin{defi}[$N$-function]\label{def:Nf} Suppose $\Omega\subset\rn$ is an open bounded set. A~function   $M:\Omega\times\rn\to\r$ is called an $N$-function if it satisfies the
following conditions:
\begin{enumerate}
\item $ M$ is a Carath\'eodory function (i.e. measurable with respect to $x$ and continuous with respect to the last variable), such that $M(x,\xi) = 0$ if and only if $\xi = 0$; and $M(x,\xi) = M(x, -\xi)$ a.e. in $\Omega$,
\item $M(x,\xi)$ is a convex function with respect to $\xi$,
\item $\lim_{|\xi|\to 0}\mathrm{ess\,sup}_{x\in\Omega}\frac{M(x,\xi)}{|\xi|}=0$,
\item $\lim_{|\xi|\to \infty}\mathrm{ess\,inf}_{x\in\Omega}\frac{M(x,\xi)}{|\xi|}=\infty$.
\end{enumerate}
\end{defi}

\begin{defi}[Complementary function] \label{def:conj} 
The complementary~function $M^*$ to a function  $M:\Omega\times\rn\to\r$ is defined by
\[M^*(x,\eta)=\sup_{\xi\in\rn}(\xi\cdot\eta-M(x,\xi)),\qquad \eta\in\rn,\ x\in\Omega.\]
\end{defi}
\begin{rem}\label{rem:f*<g*}
If $f(x,\xi)\leq g(x,\xi)$, then $g^*(x,\xi)\leq f^*(x,\xi)$.
\end{rem}

\begin{rem} If $M$ is an $N$-function and $M^*$ its complementary, we have\begin{itemize}
\item the Fenchel-Young inequality \begin{equation}
\label{inq:F-Y}|\xi\cdot\eta|\leq M(x,\xi)+M^*(x,\eta)\qquad \mathrm{for\ all\ }\xi,\eta\in\rn\mathrm{\ and\ }x\in\Omega.
\end{equation}
\item the generalised H\"older's inequality \begin{equation}
\label{inq:Holder}
\left|\int_{\Omega} \xi\cdot\eta\,dx\right|\leq 2\|\xi\|_{L_M }\|\eta\|_{L_{M^*} }\quad \mathrm{for\ all\ }\xi\in L_M(\Omega;\rn),\eta\in L_{M^*}(\Omega;\rn).
\end{equation}
\end{itemize}
\end{rem}

\begin{lem}\label{lem:M*<M} Suppose $M$ and $A$ are such that (A2) is satisfied, then 
\begin{equation*}
\int_\Omega M^*(x,A(x,\eta))dx\leq \frac{2}{c_A}\int_\Omega M\left(x, \frac{2}{c_A}\eta \right)dx\quad\text{for}\quad \eta\in L^\infty(\Omega;\rn).\end{equation*} 
\end{lem}
\begin{proof}
Since $M^*$ is convex, $M^*(x,0)=0$ and $c_A\in (0,1]$, we notice that
\[M^*\left(x,\frac{c_A}{2}A\left(x,\eta\right)\right)\leq \frac{c_A}{2}M^*(x,A(x,\eta)).\]
Taking this into account together with~(A2) and~\eqref{inq:F-Y} we have
\[\begin{split}
c_A\left(M(x,\eta)+M^*(x,A(x,\eta))\right)\leq \frac{c_A}{2} A(x,\eta)\cdot \frac{2}{c_A}\eta &\leq M\left(x,\frac{2}{c_A}\eta\right)+M^*\left(x,\frac{c_A}{2}A(x,\eta)\right)\leq\\ &\leq M\left(x,\frac{2}{c_A}\eta\right)+\frac{c_A}{2}M^*\left(x,A(x,\eta)\right).\end{split}\]
We can ignore $M(x,\eta)>0$ on the left-hand side above, rearrange the remaining terms and integrate both sides over $\Omega$ (cf.~\eqref{LinfinLM}) to get the claim.
\end{proof}

\begin{rem}\label{rem:2ndconj} For any function $f:\r^M\to\r$ the second conjugate function $f^{**}$ is convex and $f^{**}(x)\leq f(x)$. In fact,  $f^{**}$ is a convex envelope of $f$, namely it is the biggest convex function smaller or equal to~$f$.
\end{rem}

\begin{lem}\label{lem:Mass} Suppose  a cube $\Qd$ is an arbitrary one defined in (M) with $\delta_0=1/(8\sqrt{N})$ and function $M:\rn\times[0,\infty)\rightarrow[0,\infty)$  is log-H\"older continuous, that is there exist constants $a_1>0$ and $b_1\geq 1$, such that for all $x,y\in\Omega$ with $|x-y|\leq \frac{1}{2}$ and all $\xi\in\rn$ we have~\eqref{M2'}.  Let us consider function  $ \Mjd $ given by~\eqref{Mjd} and its greatest convex minorant $(\Mjd)^{**}$. Then there exist constants $a,c>0$, such that~\eqref{M2} is satisfied. 
\end{lem}
\begin{proof}[Proof. cf.~\cite{martin}]First, we fix an arbitrary $y\in Q^\delta_j$ and note that
		\begin{equation}\label{Quotient}
			\frac{M(y,\xi)}{(M^\delta_j)^{**}(\xi)}=\frac{M(y,\xi)}{M^\delta_j(\xi)}\frac{M^\delta_j(\xi)}{(M^\delta_j)^{**}(\xi)}.
		\end{equation}
		We estimate separately both quotients on the right hand side of the latter equality. By continuity of $M$ we find $\bar{y}\in \widetilde{Q}^\delta_j$ such that $M^\delta_j(\xi)=M(\bar{y},\xi)$. Then using condition~\eqref{M2'} and the fact that $|y-\bar{y}|\leq 3\delta\sqrt{d}<\frac{1}{2}$ we get
		\begin{equation}\label{FirQuoEst}
			\frac{M(y,\xi)}{M(\bar y,\xi)}\leq \max\{\xi^{-\frac{a_1}{\log|y-\bar y|}}, b_1^{-\frac{a_1}{\log|y-\bar y|}}\}\leq \max\{\xi^{-\frac{a_1}{\log(3\delta\sqrt{N})}}, b_1^{-\frac{a_1}{\log(3\delta\sqrt{N})}}\}.
		\end{equation}
		In order to estimate the second quotient in \eqref{Quotient} we observe first that if $\xi\in[0,\infty)$ is such that $M^\delta_j(\xi)=(M^\delta_j)^{**}(\xi)$ then the statement is obvious. Therefore we assume that $M^\delta_j(\xi_0)>(M^\delta_j)^{**}(\xi_0)$ at some $\xi_0$. Due to continuity of $M^\delta_j$ and $(M^\delta_j)^{**}$ there is a neighborhood $U$ of $\xi_0$ such that $M^\delta_j>(M^\delta_j)^{**}$ on $U$. Consequently, $(M^\delta_j)^{**}$ is affine on $U$. Moreover, Definition~\ref{def:Nf} implies that $m_1\leq M^\delta_j\leq m_2$, where $m_1$ and $m_2$ are convex. Therefore there are $\xi_1,\xi_2$ such that $U\subset(\xi_1,\xi_2)$, $M^\delta_j>(M^\delta_j)^{**}$ on $(\xi_1,\xi_2)$, $(M^\delta_j)^{**}(\xi_i)=M^\delta_j(\xi_i)$, $i=1,2$ and $(M^\delta_j)^{**}$ is an affine function on $[\xi_1,\xi_2]$, i.e.   
\begin{equation}\label{ConvexificationIsAffine}
	(M^\delta_j)^{**}(t\xi_1+(1-t)\xi_2)=tM^\delta_j(\xi_1)+(1-t)M^\delta_j(\xi_2),\qquad\text{for}\quad t\in[0,1].
\end{equation}
We note that we consider $\xi_1>0$, because it follows that $0=M^\delta_j(0)=(M^\delta_j)^{**}(0)$. Now, thanks to the continuity of $M$ we find $y_i\in\wt{Q}^\delta_j$ such that $M^\delta_j(\xi_i)=M(y_i,\xi_i)$, $i=1,2$. Consequently, it follows from \eqref{ConvexificationIsAffine} that
\begin{equation}\label{ConvexificationIsAffineII}
	(M^\delta_j)^{**}(t\xi_1+(1-t)\xi_2)=tM(y_1,\xi_1)+(1-t)M(y_2,\xi_2).
\end{equation}
Denoting $\tilde\xi=t\xi_1+(1-t)\xi_2$ we get
\begin{equation}\label{QuoBiConEst}
	\frac{M^\delta_j\left(\tilde\xi\right)}{(M^\delta_j)^{**}\left(\tilde\xi\right)}\leq \frac{M\left(y_2,\tilde\xi\right)}{tM(y_1,\xi_1)+(1-t)M(y_2,\xi_2)}\leq\frac{tM(y_2,\xi_1)+(1-t)M(y_2,\xi_2)}{tM(y_1,\xi_1)+(1-t)M(y_2,\xi_2)}.
\end{equation}
Next, we observe that the definition of $M^\delta_j$ implies $M(y_1,\xi_1)=M^\delta_j(\xi_1)\leq M(y_2,\xi_1)$. We can assume without loss of generality that
\begin{equation}\label{MXiOneIneq}
	M(y_1,\xi_1)< M(y_2,\xi_1)
\end{equation}
because for $M(y_1,\xi_1)= M(y_2,\xi_1)$ inequality \eqref{QuoBiConEst} implies $M^\delta_j\leq(M^\delta_j)^{**}$ on $[\xi_1,\xi_2]$. Since we have always $M^\delta_j\geq(M^\delta_j)^{**}$ we arrive at $M^\delta_j=(M^\delta_j)^{**}$ on $[\xi_1,\xi_2]$.

Let us consider a function $h:[0,1]\rightarrow\R$ defined by
\begin{equation*}
	h(t)=\frac{tM(y_2,\xi_1)+(1-t)M(y_2,\xi_2)}{tM(y_1,\xi_1)+(1-t)M(y_2,\xi_2)}.
\end{equation*}
Then we compute
\begin{equation*}
	h'(t)=\frac{(M(y_2,\xi_1)-M(y_1,\xi_1))M(y_2,\xi_2)}{(t(M(y_1,\xi_1)-M(y_2,\xi_2))+M(y_2,\xi_2))^2}.
\end{equation*}
Obviously, we have $h'>0$ on $(0,1)$ due to \eqref{MXiOneIneq}. Therefore the maximum of $h$ is attained at $t=1$, which implies
\begin{equation}
	\frac{M^\delta_j\left(\tilde\xi\right)}{(M^\delta_j)^{**}\left(\tilde\xi\right)}\leq\frac{M(y_2,\xi_1)}{M(y_1,\xi_1)}.
\end{equation}
Next, we apply condition~\eqref{M2'} and $\xi_1\leq\tilde\xi$ to infer
\begin{equation}\label{SecQuoEst}
	\frac{M^\delta_j\left(\tilde\xi\right)}{(M^\delta_j)^{**}\left(\tilde\xi\right)}\leq \max\{\xi_1^{\frac{-a_1}{\log|y_2-y_1|}}, b_1^{\frac{-a_1}{\log|y_2-y_1|}}\}\leq \max\{\xi^{\frac{-a_1}{\log|y_2-y_1|}}, b_1^{\frac{-a_1}{\log|y_2-y_1|}}\}\leq\max\{\xi^{\frac{-a_1}{\log(4\delta\sqrt{N})}}, b_1^{\frac{-a_1}{\log(4\delta\sqrt{N})}}\}
\end{equation}
since $y_1,y_2\in\tilde{Q}^\delta_j$ implies $|y_1-y_2|\leq 4\delta\sqrt{N}<\frac{1}{2}$. Combining \eqref{Quotient} with \eqref{FirQuoEst} and \eqref{SecQuoEst} yields
\begin{equation*}
\begin{split}	\frac{M(y,\xi)}{(M^\delta_j)^{**}(\xi)}\leq \max\{\xi^{\frac{-a_1}{\log(3\delta\sqrt{N})}}, b_1^{\frac{-a_1}{\log(3\delta\sqrt{N})}}\}\cdot \max\{\xi^{\frac{-a_1}{\log(4\delta\sqrt{N})}}, b_1^{\frac{-a_1}{\log(4\delta\sqrt{N})}}\}\leq \max\{\xi^{\frac{-2a_1}{\log(4\delta\sqrt{N})}}, b_1^{\frac{-2a_1}{\log(4\delta\sqrt{N})}}\}\\
\leq  \xi^{\frac{-2a_1}{\log(4\delta\sqrt{N})}}+ b_1^{\frac{-2a_1}{\log(4\delta\sqrt{N})}}\leq c \left(1+ |\xi|^{-\frac{a}{\log(b\delta )}} \right),\end{split}
\end{equation*}
which is the desired conclusion.
\end{proof}

\begin{defi}[$\Delta_2$-condition]\label{def:D2}
 We say that an $N$-function $M:\Omega\times\rn\to\r$ satisfies $\Delta_2$ condition if for a.e. $x\in\Omega$, there exists a constant $c>0$ and nonnegative integrable function $h:\Omega\to\r$ such that
\begin{equation}
\label{D2} M(x,2\xi)\leq cM(x,\xi)+h(x).
\end{equation}
\end{defi}

\section*{Appendix B}
We have two equivalent definitions of modular convergence.
\begin{defi}[Modular convergence]\label{def:convmod}
We say that a sequence $\{\xi_i\}_{i=1}^\infty$ converges modularly to $\xi$ in~$L_M(\Omega;\rn)$ (and denote it by $\xi_i\xrightarrow[i\to\infty]{M}\xi$), if 
\begin{itemize}
\item[i)] there exists $\lambda>0$ such that
\begin{equation*}
\int_{\Omega}M\left(x,\frac{\xi_i-\xi}{\lambda}\right)dx\to 0,
\end{equation*}
equivalently
\item[ii)] there exists $\lambda>0$ such that 
\begin{equation*}
 \left\{M\left(x,\frac{\xi_i}{\lambda}\right)\right\}_i \ \text{is uniformly integrable in } L^1(\Omega)\quad \text{and}\quad \xi_i\xrightarrow[]{i\to\infty}\xi \ \text{in measure};
\end{equation*}
\end{itemize}
\end{defi}

\begin{defi}[Biting convergence]\label{def:convbiting}
Let $f_n,f\in  L^1(\Omega)$ for every $n\in\N$. We say that a sequence $\{f_n\}_{n=1}^\infty$ converges in the sense of biting to $f$ in~$L^1(\Omega)$ (and denote it by $f_n\xrightarrow[]{b}f$), if  there exists a sequence of measurable $E_k$ -- subsets of $\Omega$, such that $\lim_{k\to\infty} |E_k|=0$, such that for every $k$ we have $f_n\to f$ in $L^1(\Omega\setminus E_k)$.
\end{defi}

\begin{defi}[Uniform integrability] We call a sequence  $\{f_n\}_{n=1}^\infty$ of measurable functions $f_n:\Omega\to \rn$ 
 uniformly integrable if
\[\lim_{R\to\infty}\left(\sup_{n\in\mathbb{N}}\int_{\{x:|f_n(x)|\geq R\}}|f_n(x)|dx\right)=0,\] 
equivalently (cf.~\cite{pgasgaw-stokes}) if
\begin{equation}
\label{uni-int-con2}
\forall_{\ve>0}\quad\exists_{\delta>0}\qquad \sup_{n\in\mathbb{N}}\int_\Omega \left(|f_n(x)|-\frac{1}{\sqrt{\delta}}\right)_+ dx\leq\ve,
\end{equation}
where we denote the positive part of function $f$ by $(f(x))_+:=\max\{f(x),0\}$.
 \end{defi}

We use the following results.
\begin{lem}[Modular-uniform integrability,~\cite{gwiazda2}]\label{lem:unif}
Let $M$ be an $N$-function and $\{f_n\}_{n=1}^\infty$ be a sequence of measurable functions such that $f_n:\Omega\to \rn$ and $\sup_{n\in\N}\int_\Omega M(x,f_n(x))dx<\infty$. Then the sequence $\{f_n\}_{n=1}^\infty$ is uniformly integrable.
\end{lem}

\begin{lem}[Density of simple functions, \cite{Musielak}]\label{lem:dens}
Suppose~\eqref{ass:M:int}. Then the set of simple functions integrable on $\Omega$ is dense in $L_M(\Omega)$ with respect to the modular topology.
\end{lem}
The above result can be obtained by the method of the proof of~\cite[Theorem~7.6]{Musielak}.

We need the following consequence of the Chacon Biting Lemma, \cite[Lemma~6.9]{pedr}.
\begin{theo}\label{theo:bitinglemma}Let $f_n\in   L^1(\Omega)$ for every $n\in\N$,   $f_n(x)\geq 0$ for every $n\in\N$ and a.e. $x$ in $\Omega$. Moreover, suppose $f_n\xrightarrow[]{b}f$ (cf.~Definition~\ref{def:convbiting}) and $\limsup_{n\to\infty}\int_\Omega f_n dx\leq \int_\Omega f dx.$ Then  $f_n\xrightharpoonup{}f$ in $L^1(\Omega)$ for $n\to\infty$.
\end{theo}

\begin{theo}[The Vitali Convergence Theorem]\label{theo:VitConv} Let $(X,\mu)$ be a positive measure space, $\mu(X)<\infty $, and $1\leq p<\infty$. If $\{f_{n}\}$ is uniformly integrable in $L^p_\mu$,   $f_{n}(x)\to f(x)$ in measure  and $|f(x)|<\infty $  a.e. in $X$, then  $f\in  {L}^p_\mu(X)$
and  $f_{n}(x)\to f(x)$ in  ${L}^p_\mu(X)$.
\end{theo} 

\begin{theo}[The Dunford-Pettis Theorem]\label{theo:dunf-pet}
A sequence $\{f_n\}_n$ is uniformly integrable in $L^1(\Omega)$ if and only if it is relatively compact in the weak topology.
\end{theo}

\begin{lem} \label{lem:TM1}
 Suppose $w_n\xrightharpoonup[n\to\infty]{}w$ in $L^1(\Omega)$, $v_n,v\in L^\infty(\Omega)$, and $v_n\xrightarrow[n\to\infty]{a.e.}v$. Then \[\int_\Omega w_n v_n\,dx \xrightarrow[n\to\infty]{}\int_\Omega w v\,dx.\]
\end{lem}

\section*{Appendix C}

\begin{proof}[Proof of Theorem~\ref{theo:approx}]The proof is divided into four steps. We start with the case of star-shaped domain and then, in the fourth step, we turn to any Lipschitz domain.

\medskip 

\textbf{Step 1.} Let us assume, that $\Omega$ is a star-shape domain with respect to the ball $B(0, r)$ (i.e. with respect to any point of this ball). For $0 < \delta < r/4$, we set $\kappa_\delta=1-\frac{2\delta}{r}$. It holds that
\[\kappa_\delta \Omega + \delta B(0, 1) \subset \Omega.\]
For a measurable function $\xi:\rn\to\rn$ with $\mathrm{supp}\,\xi\subset\Omega$, we define 
\begin{equation}
\label{xid}\xi_\delta(x) = \int_\Omega \vr_\delta( x-y)\xi (\kappa_\delta y)dy=
  \int_{B(0,\delta)} \vr_\delta(y)\xi (\kappa_\delta (x-y))dy,
\end{equation} 
where $ \vr_\delta(x)=\vr(x/\delta)/\delta^N$ is a standard regularizing kernel on $\rn$  (i.e. $\vr\in C^\infty(\rn)$,
$\mathrm{supp}\,\vr\subset\subset B(0, 1)$ and $\iO \vr(x)dx = 1$, $\vr(x) = \vr(-x)\geq 0$). Let us notice that $\xi_\delta\in C_0^\infty(\rn;\rn)$.

\medskip

\textbf{Step 2.} We show that the family of operators $(\xi_\delta)_\delta$ is uniformly bounded from $L_M(\Omega;\rn)$ to $L_M(\Omega;\rn)$. Without loss of generality we assume
\begin{equation}
\label{xi<1}\|\xi\|_{L^1(\Omega;\rn)}\leq 1.
\end{equation}
We have to show that\begin{equation}
\label{unifMxid}\iO M(x,\xi_\delta(x))dx\leq C\iO M(x,\xi (x))dx
\end{equation}
for every suffciently small $\delta$.

We consider $M_j^\delta(\xi)$ given by~\eqref{Mjd} and $\Mss$, see~Remark~\ref{rem:2ndconj}. Since $M(x,\xi_\delta(x))=0$ whenever $\xi_\delta(x)=0$, we have
\begin{equation}
\label{M:div-mult}\begin{split}
\iO M(x,\xi_\delta(x))dx=\sum_{j=1}^{N_\delta} \iQd M(x,\xi_\delta(x))dx=\\=\sum_{j=1}^{N_\delta} \iQdn \frac{M(x,\xi_\delta(x))}{\Msdx}{\Msdx}dx.\end{split}
\end{equation}

Our aim is to show now the following uniform bound
\begin{equation}
\label{M/M<c}\frac{M(x,\xi_\delta(x))}{\Msdx}\leq c
\end{equation}
for  sufficiently small $\delta>0$, $x\in\Qd\cap\Omega$ with $c$ independent of $\delta,x$ and $j$.  Let us fix an arbitrary cube and take $x\in \Qd$. For sufficiently small $\delta$ (i.e. $\delta< \delta_1:=\min\{ {r}/{4},\delta_0\}$), due to~\eqref{M2}, we obtain \begin{equation}
\label{M/M<xi}\frac{M(x,\xi_\delta(x))}{\Msdx} 
\leq   c \left(1+ |\xi_\delta(x)|^{-\frac{a}{\log(b\delta )}} \right).
\end{equation}

To estimate the right--hand side of~\eqref{M/M<xi} we consider~\eqref{xid}. Denote \[K=\sup_{B(0,1)}|\vr(x)|.\]
Note that for any $x,y\in\Omega$   and each $\delta>0$  we have
\[\vr_\delta(x-y)\leq  {K}/{\delta^N}.\]
Therefore, taking into account~\eqref{xi<1} we get
\begin{equation}
\label{xidest}\begin{split}|\xi_\delta(x)|& = \left| \int_\Omega \vr_\delta( x-y)\xi (\kappa_\delta y)dy\right|\\&\leq \frac{K}{\delta^N }  \int_{\Omega}  |\xi (\kappa_\delta y)|dy \leq \frac{K}{\delta^N\kappa_\delta}\|\xi\|_{L^1(\Omega;\rn)}\leq \frac{2K}{\delta^N}.\end{split}
\end{equation}
Note that $(2 K)^{-a/\log (b\delta )}\leq (4 K)^{-a/\log (b\delta_0 )}$ and \[\left|   {\delta^N} \right|^{ \frac{a}{\log (b\delta )}}=\exp \frac{aN \log  \delta}{\log (b\delta )},\]
which is bounded for $\delta\in [0,\delta_0]$. We combine this with~\eqref{M/M<xi} and~\eqref{xidest} to get\begin{equation}
\label{M/M<bezxi}\frac{M(x,\xi_\delta(x))}{(M_j^\delta(\xi_\delta(x)))^{**}} \leq c \left(1+ \left|4\frac{K}{\delta^N }\right|^{-\frac{a}{\log(b\delta )}} \right)\leq c.
\end{equation}
Thus, we have obtained~\eqref{M/M<c}. Now, starting from~\eqref{M:div-mult}, noting~\eqref{M/M<c}  and the fact that $\Mss$=0 if and only if $\xi=0$, we observe \[
\begin{split}
\iO M(x,\xi_\delta(x))dx &=\sum_{j=1}^{N_\delta} \iQdn \frac{M(x,\xi_\delta(x))}{\Msdx}{\Msdx}dx\leq \\
&\leq c\sum_{j=1}^{N_\delta} \iQdn {\Msdx}dx\leq\\
&\leq c\sum_{j=1}^{N_\delta} \iQd \   {{\Msd}\left( \int_{B(0,\delta)} \vr_\delta(y)\xi (\kappa_\delta (x-y))dy\right)}\mathds{1}_{\Qd\cap\Omega}(x) dx\leq\\
&\leq c\sum_{j=1}^{N_\delta} \int_\rn  \  {{\Msd}\left( \int_{B(0,\delta)} \vr_\delta(y)\xi (\kappa_\delta (x-y))\mathds{1}_{\Qd\cap\Omega}(x)dy\right)} dx\leq\\
&\leq  c\sum_{j=1}^{N_\delta} \int_\rn  {{\Msd}\left( \int_{B(0,\delta)} \vr_\delta(y)\xi (\kappa_\delta (x-y))\mathds{1}_{\tQd\cap\Omega}(x-y)dy\right)} dx.\end{split}
\]
Note  that by applying the Jensen inequality  the right-hand side above can be estimated by the following quantity
\[ \begin{split} 
& \quad \ c\sum_{j=1}^{N_\delta} \int_\rn \int_{\rn} \vr_\delta(y) {{\Msd}\left( \xi (\kappa_\delta (x-y))\mathds{1}_{\tQd\cap\Omega}(x-y) \right)} dy\,dx\leq\\
& \leq c \| \vr_\delta\|_{L^1({B(0,\delta);\rn})}\sum_{j=1}^{N_\delta}\int_\rn {{\Msd}\left( \xi (\kappa_\delta z)\mathds{1}_{\tQd\cap\Omega}(z) \right)}  dz\leq\\
&\leq c  \sum_{j=1}^{N_\delta}  \int_{\tQd\cap\Omega} {{\Msd}\left( \xi (\kappa_\delta z) \right)} dz.\end{split}
\] 
We applied inequality for convolution, boundedness of $\vr_\delta$, once again  the fact that $\Mss$=0 if~and only if~$\xi=0$. Then, by the definition of  $M_j^\delta(\xi)$, i.e.~\eqref{Mjd} and properties of $\Mss$, see~Remark~\ref{rem:2ndconj}, we realize that
\[ \begin{split} c  \sum_{j=1}^{N_\delta}  \int_{\tQd\cap\Omega} {{\Msd}\left( \xi (\kappa_\delta z) \right)} dz&\leq c'\sum_{j=1}^{N_\delta}  \int_{{\kappa_\delta}\tQd} {M\left(x, \xi (x) \right)}  dx\leq c'\sum_{j=1}^{N_\delta}  \int_{{2}\tQd} {M\left(x, \xi (x) \right)}  dx\leq \\&\leq C\int_\Omega {M\left(x, \xi (x) \right)} dx.\end{split}
\]
The last inequality above stands for computation of~a~sum taking into account the measure of~repeating parts of cubes.

We get~\eqref{unifMxid} by summing up the estimates of this step.

\medskip

{\bf Step 3.} Fix arbitrary $\vp\in V_0^M$ and recall definition of the cadidate for approximating family~\eqref{xid}. We are going to show that (still in the case of star-shape domains) it holds that
$$\int_\Omega M\left(x, \frac{(\nabla \vp)_\delta- \nabla \vp}{\lambda}\right) dx \xrightarrow[]{\delta\to 0} 0. $$

Fix $\sigma$ to be specified later and recall $C$ from~\eqref{unifMxid}. By Lemma~\ref{lem:dens} and  continuity of $M$ we can choose family of measurable sets  $\{ E_j \}_{j=0}^n$  such that $\bigcup_{j=0}^n E_j = \Omega$ and a simple vector valued function 
\[E^n(x)=\sum_{j=0}^n \mathds{1}_{E_j}(x) \va_{j}(x),\]
such that
	\begin{equation}\label{IE:aw15} 
	\int_\Omega M\left( x,  \frac{ E^n - \nabla \vp }{\frac{1}{3}\lambda} \right) dx < 
	\frac{\sigma}{C}.
	\end{equation}   Then by \eqref{unifMxid} we have 
	\begin{equation}\label{IE:aw16} 
	\int_\Omega M\left( x, 
	\frac{ (\nabla \vp -E^{n}  )_\dep  }{ \frac{1}{3} \lambda } \right) \,dx=\int_\Omega M\left( x, 
	\frac{ (\nabla \vp)_\dep -(E^{n})_\dep  }{ \frac{1}{3} \lambda } \right) \,dx < 
	\sigma.
	\end{equation}   Convexity of $M(x, \cdot)$ implies 
	\begin{equation*}
	\begin{split}
	&\int_\Omega 
	 M \left( x, \frac{ (\nabla \vp)_{\dep} - \nabla \vp }{ \lambda }\right) \,dx =\\
	& = \int_\Omega 
	M \left( x, \frac{ (\nabla \vp)_{\dep} -(E^n)_{\dep} 
	+  (E^n)_{\dep} -E^n +E^n - \nabla \vp}{ \lambda }\right) \,dx\\
	& \leq
	\frac{1}{3} \int_\Omega M\left( x, 
	\frac{ (\nabla \vp)_{\dep} -  (E^n)_\dep  }{ \frac{1}{3} \lambda } \right) \,dx 
	+ \frac{1}{3} \int_\Omega M\left( x, 
	\frac{  (E^n)_{\dep} -E^n }{
	\frac{1}{3} \lambda } \right) \,dx \\
	& + \frac{1}{3} \int_\Omega M\left( x,  \frac{ E^n - \nabla \vp }{\frac{1}{3}\lambda} \right) \,dx .
	\end{split}
	\end{equation*}
Since we have already estimated the first and the last expression  on the right-hand side above, let us concentrate on the second one.
 The Jensen  inequality and then the Fubini theorem lead to
	\begin{equation}\label{IE:aw17}
	\begin{split}
	 & \int_\Omega M\left( x, 
	\frac{  (E^n)_{\dep} -E^n }{
	\frac{1}{3} \lambda } \right) \,dx\\
	&=\int_\Omega M\left( x, 
	 \frac{ \sum_{j=0}^n (\mathds{1}_{E_j}(x) \va_j(x))_{\dep} -\sum_{j=0}^n \mathds{1}_{E_j}(x) \va_j(x) }{
	\frac{1}{3} \lambda } \right) \,dx\\
	& = \int_\Omega M \left( x, \frac{3}{ \lambda} \int_{B(0,\delta)} \varrho_\delta(y) \sum_{j=0}^n [   \mathds{1}_{E_j}(\kappa_\delta(x -  y)) \va_j(\kappa_\delta(x -  y)) - \mathds{1}_{E_j}(x) \va_j (x) ]\,dy \right) \,dx
	\\ &
	\leq 
	\int_{B(0,\delta)} \varrho_\delta(y)  \left( \int_\Omega M \left( x, \frac{3}{\lambda} \sum_{j=0}^n [   \mathds{1}_{E_j}(\kappa_\delta(x -  y)) \va_j(\kappa_\delta(x -  y)) - \mathds{1}_{E_j}(x) \va_j (x) ] \right) \,dx \right) \,dy.
	\end{split}
	\end{equation}
Using the continuity of the shift operator in $L^1$ we observe that poinwisely 
	\[\frac{3}{\lambda} \sum_{j=0}^n [   \mathds{1}_{E_j}(\kappa_\delta(x -  y)) \va_j (\kappa_\delta(x -  y))- \mathds{1}_{E_j}(x) \va_j (x) ] \xrightarrow[]{\dep\to 0} 0.\] 
	Moreover, note that
	$$
	 M \left( x, \frac{3}{\lambda} \sum_{j=0}^n [   \mathds{1}_{E_j}(\kappa_\delta(x -  y)) \va_j (\kappa_\delta(x -  y))- \mathds{1}_{E_j}(x) \va_j (x) ] \right) \leq \sup_{ |\vec\eta|=1}M \left( x, \frac{3}{ \lambda} \sum_{j=0}^n| \va_j| \vec\eta \right)  <\infty
$$
and
the Lebesgue Dominated Convergence Theorem provides the right-hand side of \eqref{IE:aw17} converges to zero as $\delta\to 0$.

To sum up, regarding to arbitrariness of $\sigma >0$ in~\eqref{IE:aw15} and~\eqref{IE:aw16}, and to the convergence of the second term we get the claim.

\medskip

{\bf Step 4.} If $\Omega$ is a bounded Lipschitz domain in~$\rn$, then there exists a finite family of open sets
$\{\Omega_i\}_{i\in I}$ and a finite family of balls $\{ B^i\}_{i\in I}$ such that 
$$\Omega=\bigcup\limits_{i\in I}\Omega_i$$
and every set $\Omega_i$ is star-shaped with respect to ball $B^i$ of radius $r_i$ (see e.g. \cite{Novotny}). 
Let us  introduce the partition of unity $\theta_i$ with
 for $x\in\Omega$. 
Then one can decompose function $\vp$ in the following way
	$$\vp(x) = \sum_{i\in I} (\theta_i \vp )(x).$$
Let us notice that if $\nabla \vp \in L_M(\Omega;\rn)$ and $\vp \in L^\infty(\Omega)$, then  
$\nabla (\theta_i \vp) = (\vp \nabla \theta_i + \theta_i \nabla \vp) \in L_M(\Omega;\rn)$.
Therefore we can apply the previous arguments to every function $\theta_i \vp$ of a  support on a star-shaped domain $\Omega_i\subset\Omega$.

\end{proof}

\begin{proof}[Proof of Theorem~\ref{theo:Poincare}] The proof consist of three steps starting with the case of smooth and compactly supported functions on small cube, then turning to the Orlicz class and concluding the~claim on arbitrary bounded set.

\medskip

{\bf Step 1.} We start the proof for $u\in C_0^\infty(\Omega)$ with $\mathrm{supp} u\subset\subset [-\frac{1}{4},\frac{1}{4}]^N$. Let $u$ be extended by $0$ outside $\Omega$ and $\oN=(1,\dots,1)\in\rn.$ Note that
\[u(x)=\int_{-\frac{1}{2}}^0\sum_{j=1}^N \partial_j u(x+s\oN)ds=\int_0^{\frac{1}{2}}\sum_{j=1}^N \partial_j u(x+s\oN)ds\] 
and so
\[2u(x)=\int_{-\frac{1}{2}}^{\frac{1}{2}}\sum_{j=1}^N \partial_j u(x+s\oN)ds.\] 
Then we realize that for the constant $c=\sqrt{N}/2$ we have
\[ u(x)\leq \int_{-\frac{1}{2}}^{\frac{1}{2}}\frac{1}{2}\sum_{j=1}^N |\partial_j u(x+s\oN)|ds\leq \int_{-\frac{1}{2}}^{\frac{1}{2}}c \|\nabla u(x+s\oN)\|ds.\] 
Applying $m$, which is increasing, to both sides above and the Jensen inequality (note that our interval with  the Lebesgue measure is a probability space) we get
\[m(|u(x)|)\leq m\left(\int_{-\frac{1}{2}}^{\frac{1}{2}}c \|\nabla u(x+s\oN)\|ds\right)\leq \int_{-\frac{1}{2}}^{\frac{1}{2}}   m\left(c\|\nabla u(x+s\oN)\|\right)ds.\]
Integrating over $\Omega$ and changing the order of integration we obtain
\[\begin{split}
\int_\Omega m(|u(x)|)dx&\leq \int_\Omega \int_{-\frac{1}{2}}^{\frac{1}{2}}  m\left(c\|\nabla u(x+s\oN)\|\right)dsdx=   \int_{-\frac{1}{2}}^{\frac{1}{2}} \int_\Omega  m\left(c\|\nabla u(x+s\oN)\|\right) dx ds\leq\\
&\leq   \| 1\|_{L^1\left(-d,d\right)}\sup_{s\in \left(-d,d\right)} \int_\Omega m\left(c\|\nabla u(x+s\oN)\|\right) dx =  \int_\Omega m\left(c\|\nabla u(x)\|\right) dx.\end{split}\]

Since $m\in\Delta_2$, we apply~\eqref{D2} (with constant $c_{m,\Delta_2}$ and no $x$-dependence) $k$ times with the~smallest $k$, such that $c(\Omega,N)<2^k$. Then, due to monotonicity of $m$, we get 
\[\int_{\Omega_1} m(c(\Omega,N)|\nabla {u}|)dx\leq (c_{m,\Delta_2})^k\int_{\Omega_1} m(|\nabla {u}|)dx.\] 

\medskip

{\bf Step 2.} Let us consider now an open set $\wt{\Omega}$, such that $\overline{\Omega}\subset\wt{\Omega}\subset [-\frac{1}{4},\frac{1}{4}]^N$. Step~1. provides that for $u\in C_0^\infty(\wt{\Omega})$ we have \begin{equation}
\label{inq:m-poi} \|m(|u|)dx\|_{L^1(\rn)}\leq C\|m(|\nabla u|)dx\|_{L^1(\rn)}.
\end{equation}
Now, we aim at showing that for each $u\in V_0^m$ the inequality also holds. Of course, each such $u$ can be regularised by convolution  with a standard mollifier $\varrho_\frac{1}{n}$
\[u_n(x):=\varrho_\frac{1}{n} * u(x),\]
where $\frac{1}{n}<\frac{1}{2}{\rm dist}(\partial \wt{\Omega},\Omega)$. Such $u_n$ is smooth and compactly supported in $\wt{\Omega}$, so we have~\eqref{inq:m-poi} for~ $u_n$. Passing to the limit with $n\to \infty$ gives $u_n\to u$ and $\nabla u_n\to\nabla u$ a.e. in~$\rn$. Then continuity of $m$ gives 
\[ m(|u_n|)\to m(|u|) \quad\text{and}\quad m(|\nabla u_n|)\to m(|\nabla u|)\quad  \text{a.e. in}\ \rn.\]

To get the strong convergence in $L^1(\Omega)$ of the sequence, we are going to apply the Vitali Convergence Theorem (Theorem~\ref{theo:VitConv}). It suffices to show uniform integrability of the sequence via condition~\eqref{uni-int-con2}.
Function $u\in W^{1,1}(\Omega)$, so $\nabla u_n=\varrho_\frac{1}{n}*\nabla u$.
The Jensen inequality implies
\[\int_{\wt{\Omega}} m(|\nabla u_n|)dx\leq\int_{\wt{\Omega}} m(|\nabla u|)dx.\]
 Observe that $t\mapsto |m(t)-1/\sqrt{\delta}|_+$ is a convex function and the Jensen inequality implies
\[\int_{\wt{\Omega}}\left(m(|\nabla u_n|)-\frac{1}{\sqrt{\delta}}\right)_+dx\leq\int_{\wt{\Omega}}\left(m(|\nabla u|)-\frac{1}{\sqrt{\delta}}\right)_+dx.\] Moreover, $m(|\nabla u|)\in L^1(\wt{\Omega})$, so for every $\ve>0$ there exists $\delta>0$, such the right-hand side above is smaller than $\ve$, i.e. condition~\eqref{uni-int-con2} is satisfied and we get uniform integrability of  $\{m(|\nabla u_n|)\}_n$. From~\eqref{inq:m-poi} we notice that 
 $m(| u|)\in L^1(\wt{\Omega})$ and due to the same arguments the sequence $\{m(| u_n|)\}_n$ is uniformly integrable.

\medskip

{\bf Step 3.} Suppose that $\Omega$ is arbitrary bounded set containing $0$. It is contained in the cube of~the~edge $D={\rm diam} \Omega$. Then $\wt{u}(x)=u\left(4Dx\right)$ has ${\rm supp}\,\wt u\subset \Omega_1\subset \left[-\frac{1}{4},\frac{1}{4}\right]^N.$ We have
\[\int_\Omega m(|u|)dx=(4D)^N\int_{\Omega_1} m(|\wt{u}|)dx\leq (4D)^NC\int_{\Omega_1} m(|\nabla\wt{u}|)dx=C\int_{\Omega} m(4D|\nabla {u}|)dx.\]
Moreover, we estimate the right-hand side as in Step~1 in order to put a constant outside the~integral and the claim follows for such $\Omega$. To obtain it on an arbitrary domain we need only to observe that the Lebesgue measure is translation-invariant.
\end{proof}

\bibliographystyle{plain}
\bibliography{gszg-arxiv.bib}

\end{document}